\documentclass[11pt,reqno]{amsart}
\usepackage{amsmath, amsfonts, amsthm, bm}
\usepackage{amssymb}
\usepackage{diagram}
\usepackage[all,dvips]{xy}
\usepackage{cite}
\usepackage[pagewise]{lineno}
\usepackage{appendix}
\usepackage{xcolor}
\numberwithin{equation}{section}

\usepackage{graphicx}
\usepackage{hyperref}
\usepackage{color}
\usepackage{soul,xcolor} 
\usepackage{cite}
\usepackage{comment}
\usepackage{tikz-cd}
\usepackage{tikz}
\usetikzlibrary{positioning}

\usepackage{mathtools}

\usepackage[alphabetic,initials]{amsrefs}

\usepackage{float}

\theoremstyle{plain}
\newtheorem{thm}{Theorem}[section]
\newtheorem{cor}[thm]{Corollary}
\newtheorem{lem}[thm]{Lemma}
\newtheorem{prop}[thm]{Proposition}
\newtheorem{THM}{Theorem}

\theoremstyle{definition}

\newtheorem{defi}[thm]{Definition}
\newtheorem{rem}[thm]{Remark}

\DeclareMathOperator{\Hom}{\mathrm{Hom}}
\DeclareMathOperator{\rk}{\mathrm{rk}}

\DeclareMathOperator{\para}{\mathrm{par}}

\DeclareMathOperator{\Ext}{\mathrm{Ext}}
\DeclareMathOperator{\Bun}{\mathrm{Bun}}

\DeclareMathOperator{\Jac}{\mathrm{Jac}}
\DeclareMathOperator{\ugen}{\mathrm{U^{gen}}}
\DeclareMathOperator{\gr}{\mathrm{gr}}
\DeclareMathOperator{\Aut}{\mathrm{Aut}}
\DeclareMathOperator{\End}{\mathrm{End}}

\DeclareMathOperator{\rank}{\mathrm{rank}}

\newcommand{\mc}[1]{\mathcal{#1}}

\newcommand{\E}{\mathcal{E}}
\newcommand{\C}{\mathbb{C}}

\newcommand{\cO}{\mathcal{O}}
\newcommand{\Q}{\mathbb{Q}}
\newcommand{\R}{\mathbb{R}}
\newcommand{\Z}{\mathbb{Z}}
\newcommand{\PP}{\mathbb{P}}

\newcommand{\Eb}{E_{\bullet}}

\newcommand{\pind}{\mathrm{parind}}

\newcommand{\Pmin}{P^-}
\newcommand{\Pplus}{P^+}
\newcommand{\Mmin}{M^-}
\newcommand{\Mplus}{M^+}
\newcommand{\Smin}{\Sigma^-}
\newcommand{\Splus}{\Sigma^+}
\newcommand{\bmu}{\boldsymbol{\mu}}
\newcommand{\Ugen}{\mathrm{U}^{\mathrm{gen}}}

\title{Moduli of parabolic bundles on an elliptic curve}

\author{Roberto Alvarenga}
\address{ Roberto Alvarenga, São Paulo State University (UNESP), São José do Rio Preto, Brazil.}
\email{roberto.alvarenga@unesp.br}

\author{Inder Kaur}
\address{ Inder Kaur, School of Mathematics and Statistics, University of Glasgow, G12 8QQ , UK}
\email{inder.kaur@glasgow.ac.uk}

\author{Frank Loray}
\address{Frank Loray, CNRS, IRMAR, UMR 6625, Univ Rennes, F-35000 RENNES, FRANCE}
\email{frank.loray@univ-rennes.fr}

\begin{document}

\begin{abstract}
A lot is known about the moduli space of parabolic bundles over curves of genus $g\geq 2$, but the lower genus cases are notably different. The goal of this article is to study the geometry of the moduli space of semistable parabolic bundles of rank $3$ with trivial determinant and one marked point on an elliptic curve. We show that this moduli space is rational, give an explicit description of its geometry, its group of automorphisms and also prove a Torelli-type result. 
\end{abstract}

\maketitle
\section{Introduction}

A quasi parabolic vector bundle $(E,E_\bullet)$ on a smooth projective curve $X$ with marked points $p_1, \dots ,p_s$, is a vector bundle $E$ on $X$, with the additional data of a flag $E_\bullet$ on the fiber over each parabolic point $p_i$. To define a stability condition, Mehta and Seshadri attach weights $(\mu_1,\dots, \mu_r) \in \mathbb{R}^{r}$ to these flags and the resulting bundles are called parabolic vector bundles. One of the motivations for the construction of the moduli space of (semi)stable parabolic bundles \cite{M-S} was to generalise to curves with cusps, the Narasimhan-Seshadri correspondence between stable vector bundles on smooth projective curves and unitary representations of their fundamental groups \cite{N-S}. In analogy with the moduli spaces of vector bundles, several aspects, such as the geometry \cite{Boden-Yokogawa-rat}, the automorphism group and Torelli type theorems of moduli spaces of parabolic bundles have already been studied \cites{NFV, FJ, AG}. However, as mentioned by Tu in the introduction of his article \cite{Tu},  even for moduli spaces of vector bundles, both the techniques and the results for when $X$ is an elliptic curve are very different from the higher genus cases. The dissimilarities continue when we consider parabolic vector bundles. For example, in a recent work D. Alfaya and T. Gomez \cite{AG} proved a higher rank Torelli theorem and described the group of automorphisms of the moduli space of parabolic vector bundles for the case of arbitrary weights, arbitrary rank and arbitrary number of points,  but the curve is assumed to have genus $g \geq 4$. For specific results on elliptic curves as in \cites{BBR,NFV,FJ}, the rank is assumed to be $2$ and often only specific weights are considered. In this article we undertake an extensive study of the moduli space of semistable, trivial determinant, rank $3$ bundles, on an elliptic curve, with arbitrary weights. For simplicity, we consider a unique parabolic point. We will now  describe the geometry of this moduli space, its group of automorphisms and prove a higher-rank Torelli theorem and state how the results obtained differ from those in the higher genus cases. 

\subsection{Description of the geometry of the moduli space} Let $X$ be an elliptic curve defined over the field of complex numbers $\C$,  with structure sheaf $\mc{O}_X$, and let $p \in X$. For $\bmu = (\mu_1,\mu_2, \mu_3) \in \R^3$, denote by $\Bun_{\mc{O}_X}^{\bmu}(3, X,p)$ the moduli space of rank $3$, $\bmu$-parabolic stable bundles over $X$ with parabolic structure at $p$ and trivial determinant. Recall that in this setting, Tu proved in \cite{Tu} that the moduli space of semistable rank $3$, trivial determinant vector bundles $\Bun_{\mc{O}_X}(3, X)$ is isomorphic to $\mathbb{P}^2$ (see Theorem $\ref{thm:TuExplicitIsomorphism}$). Since we consider arbitrary weights, we begin our study with the description of the wall-chamber decomposition of the space of admissible parabolic weights. As we see in section \ref{sec:WallChamberDecomposition}, in the case of a unique point, there are only $2$ chambers and thus only two possibilities for $\Bun_{\mc{O}_X}^{\bmu}(3,X,p)$ to be non empty, which we denote by $\Mmin$ and $\Mplus$. We give in section \ref{sec:StableParabAllChambers} an explicit description of the parabolic bundles that are stable irrespective of the weights (the generic points of the moduli space): they form an open subset $\Ugen$ of the two compact moduli spaces $M^\pm$. The compactifying loci, denoted ${\Sigma^\pm}=M^\pm\setminus \Ugen$, is formed by those parabolic bundles that are stable uniquely with respect to one chamber (see Section \ref{sec:StableParabSingleChamber}). We know by \cite{M-S} that the moduli space $M^{\pm}$ is smooth and of dimension $3$. In Section \ref{sec:StabilityParabVB}, we show that $\bmu$-semistability of $(E,E_\bullet)$ implies semistability of $E$,
and deduce that forgetful Tu's modular map $\Pi^{\bmu}:\Bun_{\mc{O}_X}^{\bmu}(3, X,p)\to\Bun_{\mc{O}_X}(3,X)$ is well-defined, inducing morphisms $\Pi^\pm:M^\pm\to\PP^2$. We describe its geometry in the following result (see Theorem \ref{thm:PTP2} and Theorem \ref{thm:MminequalMplus} for proofs):

\begin{THM}\label{thmA} 
The moduli spaces $\Mmin$ and $\Mplus$ are isomorphic compactifications of $\Ugen$:

\begin{diagram}
     \Mmin &\rTo^{\Psi}_{\sim}&\Mplus\\
     \uInto&\circlearrowleft&\uInto\\
      \Ugen &\rTo^{\mathrm{id}}& \Ugen
\end{diagram}
inducing an isomorphism between compactifying loci
$\Psi\vert_{\Smin}:\Smin\stackrel{\sim}{\longrightarrow}\Splus$. 

\noindent Moreover, Tu's morphism $\Pi^\pm:M^\pm\to\PP^2$ is a (locally trivial) $\PP^1$-bundle isomorphic to 
$$\Pi:\PP(T{\PP^2)}\to\PP^2.$$
Finally, the restriction $\Pi^\pm\vert_{\Sigma^\pm}:\Sigma^\pm\to\PP^2$
is a degree $3$ cover ramifying over a cuspidal sextic. 
\end{THM}

Note that for higher genus ($g \geq 2$), it is already known that the moduli space of semistable parabolic bundles with fixed determinant is rational \cite{Boden-Yokogawa-rat} although there is no explicit description of the geometry.

\subsection{A higher rank Torelli theorem}

The classical Torelli theorem \cite{Torelli} proven in $1913$ states that a smooth projective curve can be retrieved from its moduli space of degree $0$ line bundles i.e. its principally polarised Jacobian. Since the construction of the moduli space of vector bundles in the $1960$s and then the moduli space of parabolic vector bundles in $1980$, several similar Torelli-type results have been proven i.e. that the curve is determined by the moduli space of higher-rank vector bundles \cite{Mumford-Newstead} and even those with extra structures \cites{BGM, BGHL, Biswas-Hoffmann} and singular curves \cites{Caporaso-Viviani, Basu-Dan-Kaur}. For all these results, the curve is assumed to be of genus $g\geq 2$. As mentioned earlier, in the case of parabolic bundles, Alfaya and Gomez \cite{AG} showed that for a curve of genus $g\geq 4$, it can be retrieved from the moduli space of parabolic bundles of any rank, any number of points and any weights. 

Regarding the case $g=1$, we know by \cite{Tu} that the moduli spaces
of vector bundles of fixed degree and determinant is rational, the dimension of which only depends on the rank and degree: it fails to give information about the isomorphism class of $X$. As we can see
from Theorem \ref{thmA}, again $\Bun_{\mc{O}_X}^{\bmu}(3, X,p)$ is rational, and we need additional structure to get a Torelli type result.
Inspired by \cite{NFV}, we rather consider the pair $(M^\pm,\Sigma^\pm)$
taking into account the jumping locus of those bundles that are stable only for one chamber. 
In this article, we prove the following higher rank Torelli theorem:  

\begin{THM}\label{thmB} \it
Two elliptic curves $(X,p)$ and $(X',p')$ are isomorphic 
if and only if the corresponding moduli spaces of parabolic bundles
$(M,\Sigma)$ and $(M',\Sigma')$ are isomorphic. In particular, we prove that $\Sigma$ is isomorphic to the stable ruled surface over $X$.
\end{THM}

\subsection{Description of the automorphism group of the moduli space}
The moduli space of semistable parabolic bundles admits natural automorphisms arising from usual operations on parabolic bundles, such as tensoring the parabolic bundle with a line bundle, dualizing the parabolic vector bundle, taking the pullback of the parabolic vector bundle with respect to an automorphism of the curve or taking a Hecke modification with respect to a flag. We refer automorphisms of the moduli space that are combinations of these operations as \emph{modular automorphisms} (basic transformations in \cite{AG}). It is natural to ask if these are \emph{all} the automorphisms or equivalently if any automorphism of the moduli space can be expressed as a composition of these operations. In \cite{AG}, Alfaya and Gomez prove that this is indeed the case for the moduli space of semistable parabolic bundles with fixed determinant over a curve of genus $g \geq 6$. However, their techniques fail in smaller genus. Nonetheless, in our setting of $X$ being an elliptic curve, using different techniques, we show in Theorem \ref{thm:ModularAutomorphism} that every automorphism is indeed modular and that the automorphism group is finite:

\begin{THM}\label{thmC}
Assume $(X,p)$ has no other automorphism fixing $p$ than the elliptic involution $\sigma$. Then, the automorphism group of $(M^\pm,\Sigma^\pm)$ is modular, of order $18$, generated by the $9$-group of $3$-torsion points in $\Jac(X)$ together with the dualization.  Moreover, this action is just the lift to $\PP(T{\PP^2)}$ of the automorphism group of $(\PP^2,X)$.
\end{THM}

{\bf{Acknowledgements}}: We thank Maycol Falla Luza for giving us the description (and proof) for the automorphisms of $\PP(T\PP^2)$.

\section{Rank 3 semistable vector bundles on an elliptic curve}
In this section, we review some known results about semistable vector bundles of degree $0$ on the elliptic curve $X$, their moduli space following Atiyah and Tu, their automorphisms and their flat structure. We end by constructing an analytic universal family.

\subsection{Atiyah's classification in rank 3 and degree 0}
Let $X$ denote an elliptic curve, $\mc{O}_X$ the trivial line bundle on $X$. 
Following Atiyah's seminal paper \cite{atiyah57}, there are
unique non trivial extensions 
$$\mc{E}_2 \in \Ext(\mc{O}_X,\mc{O}_X)\ \ \ \text{and}\ \ \ 
\mc{E}_3 \in \Ext(\mc{E}_2,\mc{O}_X).$$
Moreover, the unique degree $0$ subbundles are
$\mc{O}_X\subset\mc{E}_2\subset\mc{E}_3$.

Recall that for a vector bundle $E$ of rank $r$ and degree $d$, the \emph{slope} of $E$, denoted $\mu(E)$ is defined as $d/r$. We say $E$ is (slope) \emph{(semi)stable} if, for any proper subbundle $F$ of $E$, $\mu(F)(\leq) < \mu(E)$. 
Then, following Atiyah's classification \cite{atiyah57}, there are $3$ types of (strictly) semi-stable vector bundles (and no stable vector bundles) of degree $0$.

\begin{thm}[{\cite{atiyah57}}]\label{Atiyah:semistablebundles} 
For rank $3$, determinant $\mc{O}_X$, the only semi-stable bundles on $X$ are 
\begin{enumerate}
    \item $E = L_i \oplus L_{j} \oplus (L_{ij})^{-1}$ for $L_{i}, L_{j} \in \Jac(X)$.
    \item $E = L^{-2} \oplus (\mc{E}_{2}\otimes L)$, for $L \in \Jac(X)$. 
    \item $E = \mc{E}_3 \otimes L$, for $L$ a $3$-torsion line bundle. 
\end{enumerate}
\end{thm}

Note that since we require the bundles to be semistable, the line bundles $L_k$ must all be of the same degree. Furthermore, since we require the determinant to be trivial, this forces the bundles $L_{k}$ to be in the Jacobian rather than just the Picard group. 

\subsection{Tu's moduli space} 
We want to describe the moduli space of semistable vector bundles of rank $3$ and trivial determinant.
It is easy to see that the set of vector bundles of type (1) in Theorem \ref{Atiyah:semistablebundles} are parametrised by a $2$-dimensional space, type (2) by the Jacobian of the curve, and 
type (3) by a finite set of points (the $3$-torsion line bundles). We now show that in fact the three types fit into a single smooth irreducible moduli space. We will be using these identifications in later sections.

\begin{defi}
 For a semistable vector bundle $E$ of slope $\bmu$ there is a natural  \emph{Jordan--H\"older filtration}
\begin{equation*}
     0=E_0\subseteq E_1\subseteq \cdots\subseteq E_s=E
 \end{equation*}
 by subbundles $E_i$ such that the quotients $E_i/E_{i-1}$, called the \emph{Jordan--H\"older factors}, are stable vector bundles of slope $\bmu$ for all $i=1,\ldots, s$. We associate to $E$ its \emph{graded vector bundle} 
 \begin{equation*}
     \gr(E)=\bigoplus_{i=1}^s E_i/E_{i-1} 
 \end{equation*}
 and say that two semistable vector bundles $E$ and $G$ are \emph{S-equivalent} if $\gr(E)\cong \gr(G)$.
\end{defi}

\begin{rem} We know that vector bundles on a Riemann surface are topologically classified by their rank and first Chern class (i.e degree). The vector bundles $\mc{E}_{2}\otimes L$ and $L \oplus L$ which are both of rank $2$ and degree $0$ are therefore topologically isomorphic. However, since the vector bundle $\mc{E}_{2}\otimes L$ is indecomposable, but the vector bundle $L \oplus L$ is decomposable, they are not algebraically (and even holomorphically) isomorphic. But they are $S$-equivalent. To see this note that by definition for the vector bundle $\mc{E}_{2}\otimes L$ we have 
\begin{equation*}
     L \to \mc{E}_{2}\otimes L \to L.
 \end{equation*}
  Therefore $\gr{\mc{E}_{2}\otimes L} = L\oplus L$. Hence $\mc{E}_{2}\otimes L$ is S-equivalent to $L \oplus L$. Furthermore, S-equivalence is preserved by taking direct sums with fixed bundles, so $L^{-2} \oplus (\mc{E}_{2}\otimes L)$ is S-equivalent to $L^{-2} \oplus L \oplus L$.   
Similarly, one can show that $\mc{E}_{3}\otimes L$ is S-equivalent to $L \oplus L \oplus L$. We will often be working up to $S$-equivalence and would therefore be identifying the semistable bundle  $L^{-2} \oplus (\mc{E}_{2}\otimes L)$ with $L^{-2} \oplus L \oplus L$ and the semistable bundle $\mc{E}_{3}\otimes L$ with $L \oplus L \oplus L$.
\end{rem}

We denote by the $\mathrm{Bun}_{\mc{O}_X}(3,X)$ moduli space of S-equivalence classes of semistable vector bundles of rank $3$ on the elliptic curve $X$ with trivial determinant $\mc{O}_X$. 

By \cite{Tu}, we know that $\mathrm{Bun}_{\mc{O}_X}(3,X)$ is isomorphic to $\PP^{2}$. The following theorem  gives the bijection between the S-equivalence classes of semistable vector bundles and $\check{\PP}^{2}$. This can be seen geometrically as follows: 

\begin{thm}\label{thm:TuExplicitIsomorphism}(Tu's isomorphism)
Let $X$ be defined as a smooth cubic curve in $\PP^{2}$, with neutral element denoted by $\infty$. Let $\mc{L}\in\check{\PP}^{2}$ be a line, intersecting the curve $X$ in $3$ points $p_1, p_2$ and $p_{3}$ (possibly with repetitions). Then we define an isomorphism 
\begin{align*}
 \psi: \check{\PP}^{2} \stackrel{\sim}{\longrightarrow} \mathrm{Bun}_{\mc{O}_X}(3,X)\\
 \mc{L} \mapsto \bigoplus\limits_{i=1}^{3} \mc{O}_X(p_{i}-\infty)\\
\end{align*}

\begin{enumerate}
\item A generic vector bundle $E = L_1 \oplus L_2 \oplus L_3$ with $L_{i} \neq L_{j}$ corresponds to a line $\mc{L}$ that intersects the curve $X$ through distinct points $p_1$, $p_2$ and $p_3$: the alignment of the three points corresponds to the condition 
$\det(E)=L_1 \otimes L_2 \otimes L_3=\mc{O}_X$.
\item The vector bundle $L^{-2} \oplus (\mc{E}_{2}\otimes L)$ is identified with the vector bundle $L^{-2} \oplus L \oplus L$ by S-equivalence. Through Tu isomorphism, these vector bundles correspond to lines $\mc{L}$ tangent to the cubic curve $X$ at $p$, with $L=\mc{O}(p-\infty)$. 
The line $\mc{L}$ intersects the curve $X$ at another point $q$ such that $L^{-2} = \mc{O}(q-\infty)$.   
\item Similarly, the vector bundle $\mc{E}_{3} \otimes L$ is identified with the vector bundle $L \oplus L \oplus L$; they correspond to a tangent line $\mc{L}$ in a flex $p\in X$ such that $L = \mc{O}(p-\infty)$. 
\end{enumerate}
\end{thm}

From this theorem, it is convenient to refine the list of semistable vector bundles (see Theorem \ref{Atiyah:semistablebundles}) into $6$ \emph{types}:
\begin{equation}\label{eq:typesofssvb}
    \begin{matrix}
     (1) & L_1 \oplus L_2 \oplus L_3 \hfill & L_i\in\Jac(X),\ L_{i}\not= L_{j},\ L_1 \otimes L_2 \otimes L_3=\cO_X\\
     \begin{matrix}
         (2.1) \\ (2.2)
     \end{matrix}
     &
     \left.\begin{matrix}
         L^{-2}\oplus(\mc{E}_{2}\otimes L) \\ L^{-2}\oplus (L\oplus L)\hfill
     \end{matrix}\right\}
     &
     L \in \Jac(X),\ L^{ 3}\not=\mc{O}_X\hfill\\
     \begin{matrix}
         (3.1) \\ (3.2) \\ (3.3)
     \end{matrix}
     &
     \left.\begin{matrix}
         \mc{E}_{3}\otimes L\hfill \\ L\oplus(\mc{E}_{2}\otimes L)\\ L\oplus L\oplus L\hfill
     \end{matrix}\hfill\right\}\hfill
     &
     L \in \Jac(X),\ L^{ 3}=\mc{O}_X\hfill
\end{matrix}
\end{equation}

\subsection{Subbundles of degree 0}

We list all possible subbundles of degree $0$ of $E$ depending on its type.

\begin{lem} \label{lem-sameslopeandrank}
    Let $E_1, E_2$ be vector bundles of the same rank and same degree over a curve $X$. If a morphism $\Phi : E_1 \rightarrow E_2$ is an isomorphism in restriction to the fiber $E_1\vert_x$ at some point $x\in X$, then it is an isomorphism globally.
\end{lem}

\begin{proof}
    The morphism $\det(\Phi):\det(E_1)\to\det(E_2)$ must be non zero at $x\in X$. Therefore, it defines a non trivial section of $\det(E_1)\otimes\det(E_2)^*$. But this 
    latter is a line bundle of degree $0$: therefore,  it is the trivial line bundle, and $\det(\Phi)$ is non vanishing. We deduce that $\Phi$ is an isomorphism.
\end{proof}

We will need the following: 

\begin{lem}\label{lem:constantrank}
Let $E$ be semistable of rank $\le3$, and $F_1,F_2\subset E$
be subbundles, all three having degree $0$. Then $F_1+F_2$ and $F_1\cap F_2$
have constant rank, therefore defining subbundles. In particular, if $E\vert_p=F_1\vert_p\oplus F_2\vert_p$ at some point $p\in X$, then $E=F_1\oplus F_2$ globally.
\end{lem}

\begin{proof}
The proof is a case-by-case analysis discussing on the rank of $F_1$ and $F_2$. Of course, if $F_1$ or $F_2$ is trivial (either $\{0\}$ or $E$), then the statement is obvious.
We first prove the second claim of the statement. If $E\vert_p=F_1\vert_p\oplus F_2\vert_p$, then we can apply Lemma \ref{lem-sameslopeandrank} to the morphism
$$\Phi:F_1\oplus F_2\to E\ ;\ (u,v)\mapsto u+v$$
which is bijective at $p$, and deduce the claim.

Note that if $\rank(E)=2$, and $\rank(F_1)=\rank(F_2)=1$, then either $F_1=F_2$, or $E\vert_p=F_1\vert_p\oplus F_2\vert_p$ at some $p\in X$. In the second case, we can conclude $E=F_1\oplus F_2$. For both of these cases, the statement holds. Now assume $\rank(E)=3$ and $\rank(F_1)=1$ and $\rank(F_2)=2$. If $F_1\subset F_2$, then the statement is obvious; and if not, then $E\vert_p=F_1\vert_p\oplus F_2\vert_p$ at some $p\in X$ and we can also conclude as before.

 Let us now assume $\rank(E)=3$ and $\rank(F_1)=\rank(F_2)=2$ and $F_1\not=F_2$ (otherwise, the statement is obvious). Then, {\it we claim that $F_1\cap F_2$ contains a line subbundle $L$}. 
 Clearly, their intersection defines a line bundle except possibly at points $x\in X$ where $F_1\vert_x=F_2\vert_x$. We have to check that this line bundle extends at those special points $x$. In the neighborhood $U$ of such a point, in a local coordinate $z$, we can assume $E\vert_U=U\times\mathbb{C}^3$, with vertical coordinate $W=\begin{pmatrix}
    w_1\\ w_2\\ w_3
\end{pmatrix}$,  $F_1=\mathbb{C}^2\times\{0\}$ defined by $w_3=0$
and $F_2$ given by a holomorphic equation $f_1(z)w_1+f_2(z)w_2+f_3(z)w_3=0$
where $f_3(0)=0$ and $f_1,f_2$ not vanishing at the same time.
Then the line subbundle is locally defined by $f_1(z)w_1+f_2(z)w_2=w_3=0$:
the local non-vanishing holomorphic section 
$W=\begin{pmatrix}
    f_2(z)\\ -f_1(z)\\ 0
\end{pmatrix}$
is generating $L$ locally. {\it This proves the claim}.
Now, since $E$ is semistable of degree $0$, then $\deg(L)\le 0$.
The rank $2$ bundle $E'=L\otimes(E/L)$ has degree $=\deg(L)$ and contains the two rank $1$ subbundles $F_i'=L\otimes(F_i/L)$, both having degree $0$.
We note that $F_1'\vert_x=F_2'\vert_x$ if, and only if $F_1\vert_x=F_2\vert_x$ for any $x\in X$. Assuming $F_1\not=F_2$,
we have $F_1'\not=F_2'$, and the composition 
$$F_1'\to E'\to E'/F_2'$$ defines a non trivial morphism $\varphi$ between line bundles.
This implies that $\deg(F_1')\le\deg(E'/F_2')$,
and therefore $0\le\deg(L)$; we deduce $\deg(L)=0$, $\deg(F_1')=\deg(E'/F_2')$ and $\varphi$
is an isomorphism. Consequently, $F_1'\vert_x\not=F_2'\vert_x$ for all $x\in X$ and $F_1\cap F_2$ has constant rank: $L=F_1\cap F_2$.

Finally, let us assume $\rank(F_1)=\rank(F_2)=1$. In a very similar way, one can show that $F_1 + F_2$ is contained in a rank $2$ subbundle $F$, and then discuss on the morphism $F_1\oplus F_2\to F$. In fact, this case is dual to the previous one, and therefore very similar.
\end{proof}

\begin{cor}\label{cor:injectiverestriction}
    Let $E$ be semistable of rank $\le3$ and degree $0$ and let $x\in X$.
    Then, any subbundle $F\subset E$ of degree $0$ is determined by its restriction
    $F\vert_x\subset E\vert_x$ to the fiber at $x$.
\end{cor}

\begin{lem} \label{lem-degree0inrank2} Let $E$ be a semistable vector bundle in $\mathrm{Bun}_{\mc{O}_X}(2,X)$. We can describe strict subbundles $F\subset E$ of degree $0$ as follows:
\begin{itemize}
    \item if $E = L \oplus L^{-1} $, with $L^{ 2}\not\cong\cO_X$, then $F = L$ or $L^{-1}$ (one of the two factors);
     \item if $E = L \oplus L $, with $L^{ 2}\cong\cO_X$, then $F \cong L$ and there is a $1$-parameter family of such subbundles;
    \item if $E = \E_2 \otimes L$, with $L^{ 2}\cong\cO_X$, then $F = L$. 
\end{itemize}    
\end{lem}

\begin{proof}
    For $E = L_1 \oplus L_2$, let $F \subseteq E$ be a line subbundle of degree $0$. Consider $\varphi_i$ the composition of the inclusion of $F$ in $E$ with the projection $E \rightarrow L_i$. If $\varphi_i = 0$, then $F \subseteq L_j$, $j\neq i$ and since they have the same rank and slope, $F = L_j$. If $\varphi_i \neq 0$, 
    If $\varphi_i \neq 0$, then applying Lemma \ref{lem-sameslopeandrank}, we deduce that $F\cong L_i$. In the case $L_1\not=L_2$,
    we deduce that $F=L_1$ or $L_2$, otherwise we would have $\varphi_1,\varphi_2\not=0$ and conclude that $F\cong L_1,L_2$, a contradiction.
    When $L_1=L_2=L$, then we have a one-parameter family of line subbundles $E\supset F\cong L$ that are in one-to-one correspondence with lines in $E\vert_x$
    (or points in $\PP(E\vert_x)\cong\PP^1$); by Corollary \ref{cor:injectiverestriction}, there are no other line subbundle of degree $0$.

Next, let $F \subseteq \E_2 \otimes L$ be a line subbundles of degree $0$. Observe that the task is equivalent to describe the line subbundles of degree $0$ of $\E_2$. Let $F \subseteq \E_2$ be a line subbundle of degree $0$. Then,
\[ 0 \to F \to \E_2 \to F' \to 0\]
which implies that $\mathrm{Ext}^1(F', F) = H^1(X, F \check{F}')$ is nontrivial. Riemann-Roch yields $F = F'$ and $\E_2$ is the unique nontrivial extension in $\mathrm{Ext}^1(F, F)$. But we also know that $\E_2$ is the (unique) nontrivial extension in $\mathrm{Ext}^1(\cO_X, \cO_X)$. Then, we must also have
\[ 0 \rightarrow F \rightarrow \E_2 \otimes F \rightarrow F\rightarrow 0.\]
Therefore, $\E_2 \cong \E_2 \otimes F$. Since $\E_2$ is indecomposable of degree $0$, by \cite[Theorem $5,(ii)$]{atiyah57}, $F \cong \cO_X$. 
\end{proof}

\begin{prop}\label{prop:SubbundlesOfDegree0}
Let $E$ be a semistable vector bundle in $\mathrm{Bun}_{\mc{O}_X}(3,X)$. We can describe strict subbundles $F\subset E$ of degree $0$,  depending on the type of $E$ as follows:
\begin{itemize}
    \item[(1)] if $E=L_1 \oplus L_2 \oplus L_3$, with $L_i \not\cong L_j$, then $F=L_i$ or $F=L_i\oplus L_j$, $i,j=1,2,3$;
    \item[(2.1)] if $E=L^{-2} \oplus (\mc{E}_{2}\otimes L)$, with $L^{3} \not\cong \cO_X$, then $F=L^{-2},$ $L$, $\mc{E}_{2}\otimes L$ or $L^{-2}\oplus L$;
    \item[(2.2)] if $E=L^{-2}\oplus L\oplus L$, 
    with $L^{3} \not\cong \cO_X$, then $F=L^{-2}$, $L \oplus L$,
    or $F\cong\underbrace{L}_{\infty^1}$ or $\underbrace{L^{-2}\oplus L}_{\infty^1}$; 
    \item[(3.1)] if $E=\mc{E}_3 \otimes L$, then $F=L$ or $\mc{E}_2\otimes L$; 
     \item[(3.2)] if $E=L\oplus (\mc{E}_{2}\otimes L)$, then $F=L \oplus L$, or $F\cong\underbrace{L}_{\infty^1}$ or  $\underbrace{\mc{E}_{2}\otimes L}_{\infty^1}$;
    \item[(3.3)] if $E=L\oplus  L\oplus  L$, then $F\cong\underbrace{L}_{\infty^2}$ or $\underbrace{L\oplus L}_{\infty^2}$. ($2$-parameter families)
\end{itemize}
We plot $\infty^1$ or $\infty^2$ when the subbundle is not unique, but can be deformed into a $1$- or $2$-parameter family.
\end{prop}

In Figure \ref{fig:degree0bundle}, we consider for each type, the restriction to $\PP(E_p)$ of degree $0$ line subbundles (red points) and rank $2$ subbundles (blue lines)
\begin{figure}[H]
    \centering
    \includegraphics[width=0.8\linewidth]{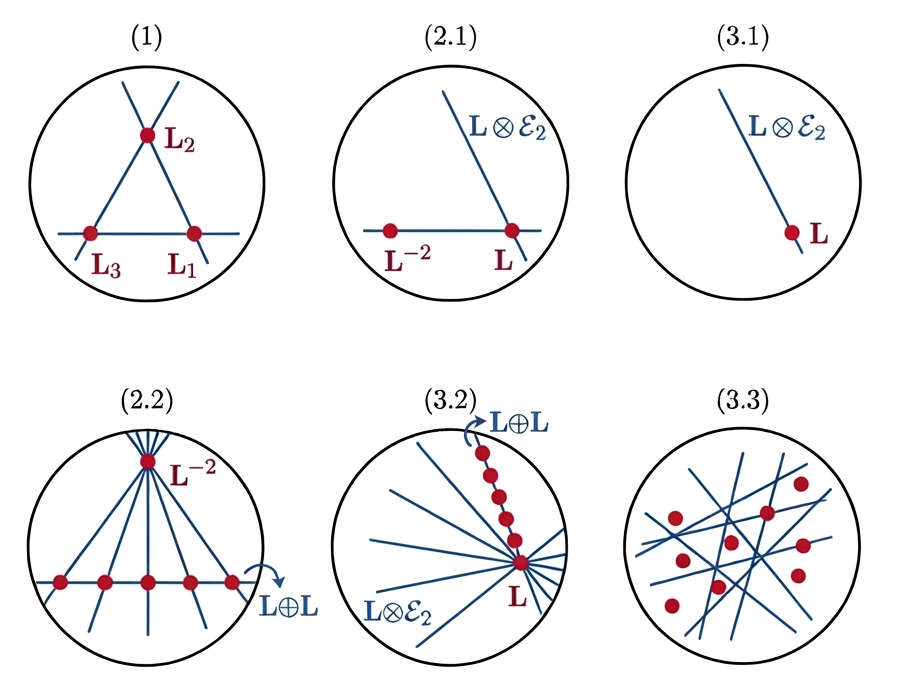}
    \caption{Degree $0$ subbundles}
    \label{fig:degree0bundle}
\end{figure}

\begin{proof}[Proof of Proposition \ref{prop:SubbundlesOfDegree0}]
It is straightforward to check that, in each case, $E$ contains at least the given subbundles of degree $0$. The only non trivial case is to check that there is indeed a $1$-parameter family of degree $0$ rank $2$  subbundles in case (3.2), i.e. that 
$\mc{E}_{2}\otimes L$ can be deformed. But it directly follows from the computation of automorphisms in Proposition \ref{prop:endomorphisms}: by setting $a=1$, $b=c=0$ and $d$ arbitrary, we get a $1$-parameter family of automorphisms deforming 
$\mc{E}_{2}\otimes L$ (and fixing the line subbundle $L\subset \mc{E}_{2}\otimes L$).

Now, let us see how to prove that there are no other degree zero subbundles. 

Let $E=L_1 \oplus L_2 \oplus L_3$, $\deg(L_i)=0$, and let $F\subset E$ be a line subbundle of degree $0$, $F\not=L_i$, $i=1,2,3$. 
If $F\subset L_1\oplus L_2$, then $L_1\cong L_2$ by applying Lemma \ref{lem-degree0inrank2} and we have a one-parameter family of such $F$: we are in case (2.2). If $F\not\subset L_i\oplus L_j$ for all $i,j=1,2,3$, then we can conclude
$$E\vert_x=L_1\vert_x\oplus L_2\vert_x\oplus F\vert_x=L_1\vert_x\oplus F\vert_x\oplus L_3\vert_x=F\vert_x\oplus L_2\vert_x\oplus L_3\vert_x$$
at a general point $x\in X$, and therefore
$$E=L_1\oplus L_2\oplus F=L_1\oplus F\oplus L_3=F\oplus L_2\oplus L_3$$
by Lemma \ref{lem:constantrank}; we deduce 
$$L_1\otimes L_2\otimes L_3=\det(E)=L_1\otimes L_2\otimes F=L_1\otimes F\otimes L_3=F\otimes L_2\otimes L_3$$
implying $F\cong L_i$ for all $i$, and we are in case $(3.3)$: 
degree $0$ line subbundles $F$ are in one-to-one correspondence 
with points in $\PP(E\vert_x)\cong\PP^2$.
Finally, in case (1), we conclude by contradiction that $F$ must coincide with one of the $3$ factors $L_i$.

Let $F\subset E=L_1 \oplus L_2 \oplus L_3$ be a rank $2$ subbundle of degree $0$, $F\not=L_i\oplus L_j$ for all $i,j=1,2,3$. If $F$ contains one factor, 
say $L_1\subset F$, then $F_1:=F\cap (L_2\oplus L_3)$ defines a line subbundle 
of degree $0$ by Lemma \ref{lem:constantrank}, which must be different from
$L_2,L_3$ by assumption, and we can conclude $L_2\cong L_3$ like in case (2.2).
On the other hand, if $F$ contains no factor $L_i$, $i=1,2,3$, then we similarly
conclude that 
$$E=L_1\oplus F=L_2\oplus F=L_3\oplus F$$
which implies $L_1\cong L_2\cong L_3$ by taking the determinant.

We check another example, and leave the other cases for the reader.
Let $E=L\oplus (\mc{E}_{2}\otimes L)$ be of type (3.2).
It contains a subbundle $L_0\subset \mc{E}_{2}\otimes L$
which is isomorphic to $L$, and the $1$-parameter family given by linear combinations of $L$ and $L_0$. It also contains a $1$-parameter family of rank $2$ subbundles $\cong \mc{E}_{2}\otimes L$
all containing $L_0$, that form a pencil together with $L\otimes L_0$. Now suppose $F\subset E$ is of degree $0$, distinct from previous bundles. If $F$ has rank $2$, then $F\cap(\mc{E}_{2}\otimes L)$ defines a new line subbundle $L_0'$ of degree $0$ making $\mc{E}_{2}\otimes L$ decomposable, $\cong L_0\oplus L_0'$, a contradiction. If $F$ has rank $1$, again it must be contained in a copy of $\mc{E}_{2}\otimes L$, contradiction.
\end{proof}

\begin{rem} The proof given by the description of the degree zero subbundles of $L_1 \oplus L_2 \oplus L_3$, with $\deg L_i = 0$, generalizes to any rank. i.e., given $L_i \in \Jac X, \; i=1, \ldots, n$, then any subbundle of $L_1 \oplus \cdots \oplus L_n$ is a direct sum of some $L_i's$. Moreover, another way to prove it is observing that: (i) given $\mu \in \Q$, the category $\mathrm{C}_\mu$ of semistable bundles of slope $\mu$ is an abelian category, cf. \cite[Prop. 5.3.6]{LePotier97}; (ii) in an abelian category, any subobject of a direct sum $E_1 \oplus \cdots \oplus E_n$ of simple objects $E_i$ is given by $\oplus_{i \in I} E_i$ where $I \subseteq \{1, \ldots, n\}$; (iii) follows from Atiyah \cite{atiyah57} that the set of stable bundles of slope $\mu \in \Q$ is the set of simple objects of $\mathrm{C}_\mu$; and (iv) every line bundle is a stable bundle (indeed the unique bundles that are stable in our setting).    
\end{rem}

\subsection{Automorphisms of vector bundles}

Let $\End(E)$ denotes the vector space of global endomorphisms of $E$,
and $\Aut(E)$, the group of global automorphisms.
Given a point $p\in C$, consider the morphism of Lie algebras:
$$R_p\ :\ \End(E)\to \mathrm{gl}(E\vert_p)\ ;\ \Phi\mapsto \Phi\vert_p$$
given by restriction to a fiber. For instance, when $E$ is a line bundle, 
then this morphism is bijective, and $\End(E)\cong\mathbb{C}$.

\begin{prop}\label{prop:endomorphisms}
Let $E$ be a semistable vector bundle in $\mathrm{Bun}_{\mc{O}_X}(3,X)$.
Then the morphism $R_p$ is injective and its image $G=R_p(\End(E))\subset\mathrm{gl}(E\vert_p)$ can be described as follows:
\begin{itemize}
    \item[(1)] if $E=L_1\oplus ,L_2\oplus ,L_3$, then $G=\left\{\begin{pmatrix}a_1&0&0\\0&a_2&0\\0&0&a_3\end{pmatrix},\ (a_1,a_2,a_3)\in\mathbb{C}^3\right\}$;
    \item[(2.1)] if $E=L^{-2} \oplus (\mc{E}_{2}\otimes L)$, then $G=\left\{\begin{pmatrix}a_1&0&0\\0&a_2&b\\0&0&a_2\end{pmatrix},\ (a_1,a_2,b)\in\mathbb{C}^3\right\}$;
    \item[(2.2)] if $E=L^{-2}\oplus L\oplus L$, then $G=\left\{\begin{pmatrix}a_1&0&0\\0&a_2&b\\0&c&a_3\end{pmatrix},\ (a_1,a_2,a_3,b,c)\in\mathbb{C}^5\right\}$; 
    \item[(3.1)] if $E=\mc{E}_3 \otimes L$, then $G=\left\{\begin{pmatrix}a&b&c\\0&a&b\\0&0&a\end{pmatrix},\ (a,b,c)\in\mathbb{C}^3\right\}$;  
    \item[(3.2)] if $E=L\oplus(\mc{E}_2 \otimes L)$, then $G=\left\{\begin{pmatrix}a&0&d\\c&a&b\\0&0&a\end{pmatrix},\ (a,b,c,d)\in\mathbb{C}^4\right\}$;
    \item[(3.3)] if $E=L\otimes L\otimes L$, then $G=\mathrm{gl}(\mathbb{C}^3)$. 
\end{itemize}
In cases (1), (2.1), (3.1), the $G$-invariant subspaces correspond to the restriction
at $p$-fiber of degree $0$ subbundles.
\end{prop}

\begin{proof}
The first observation is that, if a subbundle $F\subset E$ of degree $0$ cannot be deformed,
then it must be preserved by $\End(E)$, since $\End(E)$ is the Lie algebra of 
infinitesimal automorphisms. 

In case (1), each $L_i$ must be preserved
and through a convenient basis, $G$ is the diagonal Lie algebra: 
$$\End(E)=\End(L_1)\oplus\End(L_2)\oplus\End(L_3).$$
In case (2.1), we have
$$\End(E)=\End(L^{-2})\oplus\End(\mc{E}_{2}\otimes L)$$
and $\End(\mc{E}_{2}\otimes L)=\End(\mc{E}_{2})$ (twist by line bundle does not affect endomorphisms). From \cite{Maruyama}[page 92 (4")(a)],
we get $$\End(\mc{E}_{2})\cong\left\{\begin{pmatrix}a&b\\0&a\end{pmatrix},\ (a,b)\in\mathbb{C}^2\right\}.$$
In case (2.2), we have
$$\End(E)=\End(L^{-2})\oplus\End(L\oplus L)$$
and we obviously have $\End(L\oplus L)=\End(\cO_X\oplus\cO_X)\cong\mathrm{gl}(\mathbb{C}^2)$.

We provide a proof for case (3.1), inspired by \cite{Maruyama} for the rank $2$ case. We first observe that any endomorphism $\Phi$ must preserve the flag $\cO_X\subset \mc{E}_2\subset\mc{E}_3$. Next, observe that eigenvalues of $\Phi$
vary holomorphically on $X$, and therefore constant. Moreover, they are equal, otherwise
the corresponding decomposition of eigenspaces will produce a decomposition of 
$\mc{E}_3$, impossible. Now, $\mc{E}_3$ can be defined by a \v{C}ech cocycle
in $(M_{ij})_{ij}\in H^1(X,\mathrm{GL}(\cO_X))$ of the form:
$$M_{ij}=\begin{pmatrix}
    1&a_{ij}&b_{ij}\\ 0&1&c_{ij}\\0&0&1
\end{pmatrix}$$
where $a_{ij},b_{ij},c_{ij}$ define cocycles in $H^1(X,\cO_X)$ arising from 
successive extensions $\cO_X\to \mc{E}_2\to\mc{E}_3$.
Along the diagonal, we have put $1$ for the cocycle of the trivial bundle.
Since $H^1(X,\cO_X)\cong\mathbb{C}$, it is generated by a single element $(\sigma_{ij})$ and, maybe changing to an equivalent cocycle in $ H^1(X,\mathrm{GL}(\cO_X))$, we may assume
$$ M_{ij}=\begin{pmatrix}
    1&a\sigma_{ij}&b\sigma_{ij}\\ 0&1&c\sigma_{ij}\\0&0&1
\end{pmatrix}$$
for constant $a,b,c\in\mathrm{C}$. By obvious change, we can even assume in Jordan
normal form: 
$$ M_{ij}=\begin{pmatrix}
    1&\sigma_{ij}&0\\ 0&1&\sigma_{ij}\\0&0&1
\end{pmatrix}.$$
On the other hand, an endomorphism $\Phi$ of $\mc{E}_3$
is locally given by 
$$ \Phi_{i}=\begin{pmatrix}
    \alpha&\beta_i&\delta_i\\ 0&\alpha&\gamma_i\\0&0&\alpha
\end{pmatrix}$$
with $\alpha\in\mathbb{C}$ and $\beta_i,\gamma_i,\delta_i$ holomorphic.
Cocycle relation $\Phi_i M_{ij}=M_{ij} \Phi_j$ gives
$$\left\{\begin{matrix}
    \beta_i-\beta_j=0,\hfill\\
    \gamma_i-\gamma_j=0,\hfill\\
    \delta_i-\delta_j=(\gamma_j-\beta_i)\sigma_{ij}.
\end{matrix}\right.$$
We first deduce that $(\beta_i)_i$ and $(\gamma_i)_i$ define global holomorphic sections, and therefore constant scalars $\beta,\gamma\mathrm{C}$.
Then, the third identity rewrites $\delta_i-\delta_j=(\gamma-\beta)\sigma_{ij}$;
if $\gamma\not=\beta$, then $\sigma_{ij}$ is a coboundary of $\left(\frac{\delta_i}{\gamma-\beta}\right)_i$, and is therefore trivial in $H^1(X,\cO_X)$, contradiction.
Therefore, $\gamma=\beta$, and $\delta_i=\delta_j$, defining a global holomorphic function, and therefore constant. We get the normal form of the statement for $\Phi$.

In case (3.2), we proceed as before. After twisting by $L^{-1}$, which does not 
change the endomorphism group, we can assume $L=\cO_X$ with trivial cocycle,
and $E$'s cocycle given by 
$$M_{ij}=\begin{pmatrix}
    1&0&0\\ 0&1&\sigma_{ij}\\0&0&1
\end{pmatrix}.$$
Then, assuming $\Phi$ locally given by 
$$\Phi_i=\begin{pmatrix}
  a_i & b_i & c_i\\ 
  d_i & e_i & f_i\\
  g_i & h_i & k_i
\end{pmatrix},$$
the compatibility condition $\Phi_i M_{ij}-M_{ij} \Phi_j=0$ gives
$$a_i-a_j=b_i-b_j=g_i-g_j=h_i-h_j=0$$
defining global holomorphic and therefore constant functions $a,b,g,h\in\mathbb{C}$,
and
$$\left\{\begin{matrix}
    c_i-c_j=-b\sigma_{ij}\hfill\\
    d_i-d_j=g\sigma_{ij}\hfill\\
    e_i-e_j=h\sigma_{ij}\hfill\\
    k_i-k_j=-h\sigma_{ij}\hfill\\
    f_i-f_j=(k_j-e)\sigma_{ij}
\end{matrix}\right.$$
Since $\sigma_{ij}$ is a non trivial cocycle, we have $b=g=h=0$
and 
$$c_i-c_j=d_i-d_j=e_i-e_j=k_i-k_j=0$$
defining constants $c,d,e,k\in\mathbb{C}$; finally, the last equality yields 
$k=e$ and $f_i=f\in\mathbb{C}$ constant, which is the form of the statement.
Case (3.3) is obvious.
\end{proof}

\begin{cor}
    Let $E$ and $G$ be as in Proposition \ref{prop:endomorphisms}. Then 
    $\Aut(E)=G\cap\mathrm{GL}(E\vert_p)$.
\end{cor}

\subsection{Connections and flat vector bundles}

\begin{defi}
Let $E$ be a vector bundle on $X$.
\begin{enumerate}
    \item A \emph{connection} on $E$ is a $\mathbb{C}$-linear map $\nabla: E \to E\otimes \Omega^1_{X}$ satisfying the Leibniz rule \[ \nabla(f\sigma) = \sigma \otimes df + f\nabla(\sigma),\]  for (local) sections $\sigma \in \Gamma(U,E)$ and functions $f \in \gamma(U,\mc{O}_X)$ on any analytic open set $U \subset X$.
    \item A (local) $\nabla$-horizontal section of $(E,\nabla)$ is a (local) 
    section $\sigma$ such that $\nabla(\sigma)=0$.
\end{enumerate}
\end{defi}

Since $X$ has dimension one, $\nabla$ is automatically flat, i.e. locally admits a basis of horizontal sections
$(\sigma_1,\ldots,\sigma_r)$, $\nabla(\sigma_i)=0$.
We will call \emph{flat vector bundle} a pair $(E,\nabla)$ as above. We say that two flat vector bundles $(E,\nabla)$ and $(E',\nabla')$ over $X$ are \emph{isomorphic} if there exists an isomorphism of vector bundles $\Phi:E\to E'$
sending $\nabla$-horizontal sections to $\nabla'$-horizontal sections.

We recall the following result from \cite{Weil} and {\cite{atiyah_2}}.

\begin{thm}[Weil]\label{thm:Atiyah-Weil}
Let $E$ be an indecomposable vector bundle of degree $0$. Then $E$ admits a holomorphic connection if,
and only if, for a (any) maximal decomposition $E=F_1\oplus F_2\oplus\cdots\oplus F_k$, each indecomposable component $F_i$ of $E$ has a degree $0$.
\end{thm}

It immediately follows from Theorem \ref{Atiyah:semistablebundles}  
that 

\begin{cor}
    A vector bundle $E$ of degree $0$ and rank $3$ on an elliptic curve $X$ is semi-stable if, and only if, it is flat.
\end{cor}

There is a close link between endomorphisms and flat structure (i.e. connections)
on an elliptic curve: any two connections $\nabla_1,\nabla_2$ on $E$
differ by a Higgs field 
$$\Phi:=\nabla_1-\nabla_2\in H^0(X,\mc{E}nd(E)\otimes\Omega^1_X)$$
where $\mc{E}nd(E)$ is the sheaf of endomorphisms of $E$; but on $X$ elliptic,
$\Omega^1_X=\cO_X$ and $$H^0(X,\mc{E}nd(E)\otimes\Omega^1_X)=\End{E}.$$

\subsection{Monodromy representations of a flat bundle}\label{sec:monodromyofflatbundles}

A flat vector bundle $(E,\nabla)$ can locally be trivialized by analytic gauge transformation. Given a base point $x_0\in X$
and a local trivialization $\Phi:E\vert_U\to \mathbb{C}^3$, for an open analytic neighborhood $U$, then $\Phi$ admits an analytic continuation 
along any path starting from $x_0$; for a loop $\gamma$ based on $x_0$,
we denote by $\Phi^\gamma$ the new local trivialization obtained after
the analytic continuation of $\Phi$ along $\gamma$. We have 
$$\Phi^\gamma=M_\gamma \Phi$$
for a matrix $M_\gamma\in\mathrm{GL}(\mathbb{C}^3)$ that depends only on the 
isotopy class $[\gamma]$ of $\gamma$, and we inherit a \emph{monodromy representation}
$$\rho_\nabla:\pi_1(X,x_0)\to\mathrm{GL}(\mathbb{C}^3)\ ;\ [\gamma]\mapsto M_\gamma$$
Recall the Riemann-Hilbert correspondence:

\begin{thm}
    The monodromy map 
    $$\mathrm{Mon}\ :\ (E,\nabla)\mapsto \rho_\nabla$$
    induces a one-to-one correspondence between the moduli space of flat vector bundles (i.e. equipped with a holomorphic connection) and the moduli space of representations. In other words:
    \begin{itemize}
        \item{\bf Injectivity:} If $(E,\nabla)$ and $(E',\nabla')$ have the same representation up to $\mathrm{GL}(\mathbb{C}^3)$-conjugacy, then they are isomorphic;
        \item{\bf Surjectivity:} For any representation $\rho\in\Hom(\pi_1(X,x_0),\mathrm{GL}(\mathbb{C}^3))$, there exists a flat vector bundle $(E_\rho,\nabla_\rho)$ whose monodromy is $\rho$, up to $\mathrm{GL}(\mathbb{C}^3)$-conjugacy.
    \end{itemize}
\end{thm}

In (2), we emphasize that the bundle itself $E_\rho$ depends 
on the representation $\rho$.

Let $u:\tilde X\cong\mathbb{C}\to X$ be the universal cover with translation lattice $\Lambda=\mathbb{Z}+\tau\mathbb{Z}$. The fundamental group $\pi_1(X)$ is generated by loops $(\alpha,\beta)$ corresponding to $(1,\tau)$ in the lattice. A representation $\rho$
is determined by a pair $(A,B)$ of commuting matrices, namely 
$$A=\rho(\alpha)\ \ \ \text{and}\ \ \ B=\rho(\beta),$$
and we will denote 
$(E_{(A,B)},\nabla_{(A,B)})$ the corresponding flat vector bundle.
This is a useful way to construct a (topologically trivial) holomorphic vector bundle.

Let us first consider the case of line bundles. Here, $\mathrm{GL}(\mathbb{C})\cong\mathbb C^*$ and we have a map
$$\mathbb{C}^*\times\mathbb{C}^*\ni(a,b)\mapsto (E_{(a,b)},\nabla_{(a,b)})$$
which is one-to-one between $\mathbb{C}^*\times\mathbb{C}^*$ and the moduli space of flat rank $1$ vector bundle (Riemann-Hilbert correspondence). On the other hand,
we have 
$$E_{(a,b)}\cong E_{(a',b')}\ \ \ \Leftrightarrow\ \ \ \left\{\begin{matrix}
    a'&=& a e^{c}\hfill\\ b'&=& b e^{c\tau}
\end{matrix}\right.\ \text{for some}\ c\in\mathbb{C}.$$
Indeed, the difference $\phi=\nabla'-\nabla$ on a given vector bundle is 
a holomorphic $1$-form, i.e. $\phi=c\cdot dz$ where $z:\tilde X\stackrel{\sim}{\to}\mathbb{C}$, which gives
after integration that monodromy of $\nabla'$ along a loop $\gamma$ is obtained by multiplying the monodromy of $\nabla$  by $e^{c\tilde\gamma}$, where $\tilde\gamma$ a lift of $\gamma$ on $\tilde X$. 

On the other hand, one can consider the ratio $\frac{b}{a^\tau}$, which is well defined modulo 
multiplication by $e^{2i\pi\tau}$, as an element of the elliptic curve 
$$\frac{b}{a^\tau}\in\mathbb{C}^*/ \langle e^{2i\pi\tau} \rangle \cong\Jac(X).$$
The preceding discussion can be reformulated into the exact sequence
$$0\to H^0(X,\Omega^1_X)\to H^1(X,\mathbb{C}^*)\to\Jac(X)\to 0$$
that can be extracted from the long exact sequence associated to 
$0\to\mathbb{C}^*\to\cO^*_X\stackrel{d\log}{\to}\Omega^1_X\to0$. 
Here $H^1(X,\mathbb{C}^*)$ is interpreted as the moduli space of representations.
With previous notations, the maps are respectively $cdz\mapsto (e^c,e^{c\tau})\in\mathbb{C}^*\times\mathbb{C}^*\ni(a,b)\mapsto \frac{b}{a^\tau}$.

Coming back to the rank $3$ case. Given a rank $3$ flat vector bundle $(E,\nabla)$,
recall that $\nabla$ induces a \emph{trace connection} $\mathrm{tr}(\nabla)$
on $\det(E)$; with our notations, its monodromy is given by $(a,b)=(\mathrm{tr}(A),\mathrm{tr}(B))$. Since we are considering vector bundles with trivial determinant, 
it is natural to only consider connections with trivial trace, i.e. \emph{trace-free}.

\begin{prop}
    Let $(E,\nabla)$ be a trace-free rank $3$ flat vector bundle over $X$, and let 
    $(A,B)\in\mathrm{SL}(\mathbb{C}^3)$ be the corresponding monodromy. Then, we are in one of the 
    following situations up to $\mathrm{SL}(\mathbb{C}^3)$-conjugacy and permutation of $A$ and $B$:
    \begin{enumerate}
    \item[(i)] $A=\begin{pmatrix}
            a_1&0&0\\0&a_2&0\\0&0&a_3
        \end{pmatrix}$ and $B=\begin{pmatrix}
            b_1&0&0\\0&b_2&0\\0&0&b_3
        \end{pmatrix}$ with $\prod_ia_i=\prod_ib_i=1$.
        
        \noindent Then, $E$ is of type (1), (2.2) or (3.3)) depending whether the cardinality of $\left\{\frac{b_1}{a_1^\tau}, \frac{b_2}{a_2^\tau},\frac{b_3}{a_3^\tau}\right\}\subset\Jac(X)$ is $3$, $2$ or $1$, respectively.
    \item[(ii)] $A=\begin{pmatrix}
            a^{-2}&0&0\\0&a&1\\0&0&a
        \end{pmatrix}$ and $B=\begin{pmatrix}
            b^{-2}&0&0\\0&b&b_1\\0&0&b
        \end{pmatrix}$ with $a,b\not=0$. 
        
        \noindent Then $E$ is of type (2.1), (2.2), (3.2) or (3.3) depending on 
        $\frac{b}{a^\tau}\subset\Jac(X)$ and $\frac{ab_1}{b}$.
    \item[(iii)] $A=\begin{pmatrix}
            a&1&0\\0&a&1\\0&0&a
        \end{pmatrix}$ and $B=\begin{pmatrix}
            b&b_1&b_2\\0&b&b_1\\0&0&b
        \end{pmatrix}$ with $a^3=b^3=1$.

        \noindent Then $E$ is of type (3.1), (3.2), (3.1) depending on $\frac{ab_1}{b}$ and $b_2$
    \end{enumerate}
\end{prop}

\begin{proof}
    One can first observe that monodromy matrices $A$ and $B$ must commute, since $\pi_1(X)$ is abelian, and be of determinant $1$. Assuming that $A$ is in Jordan normal form, one can look for commuting matrices $B$ and try to simplify without changing $A$. Maybe permuting $A$ and $B$, we arrive at cases (i), (ii) and (iii) of the statement. 
    
For instance, if 
$A=\begin{pmatrix}a^{-2}&0&0\\0&a&1\\0&0&a\end{pmatrix}$ with $a^3\not=1$,
then 
$B=\begin{pmatrix}b&0&0\\0&b_{11}&b_{12}\\0&b_{21}&b_{22}\end{pmatrix}$; by further conjugating by a matrix $C$ of the same shape as $B$ (and therefore commuting with $A$), we can put $B$ into Jordan normal form, i.e. either diagonal and we are in case (i), or with a rank $2$ Jordan block, and after permutation we are in case (iii).

    We now explain what is the type of the bundle, depending on the monodromy. In case (i), each factor $\mathbb C e_i$ is invariant and gives rise to a flat rank one subbundle of $E$, where $(e_1,e_2,e_3)$ denotes the standard basis. Therefore, we have a decomposition $E=L_1\oplus L_2\oplus L_3$ with
    $L_i=E_{(a_i,b_i)}$ under notations of the beginning of the section. Moreover, $L_i\cong L_j$ if, and only if, 
$\frac{b_i}{a_i^\tau}$ and $\frac{b_j}{a_j^\tau}$ define the same point in the Jacobian. In case (ii), we have a decomposition $E=L\oplus F$ for subbundles $L,F$ of rank $1$ and $2$ respectively.
Then need to understand whether $F$ is decomposable or not. After tensoring by $E_{(a,b)}^{-1}$, we have to decide if 
$$\tilde A=\begin{pmatrix}1&\frac{1}{a}\\0&1\end{pmatrix}
\ \ \ \text{and}\ \ \ 
\tilde B=\begin{pmatrix}1&\frac{b_1}{b}\\0&1\end{pmatrix}$$
is the monodromy of $\mathcal E_2$ or the trivial bundle. 
If we assume $E_{\tilde A,\tilde B}$ is trivial, then we can
make explicit the flat structure. Precisely, it is given by a connection of the form
$\nabla=d+\begin{pmatrix}\alpha&\beta\\\gamma&\delta\end{pmatrix}dz$ where 
$dz$ is the holomorphic $1$-form on $X$, and $\alpha,\beta\gamma,\delta\in\mathbb C$ are constants. But the line subbundle generated
by $e_1$ is $\nabla$-invariant and the restriction of $\nabla$ must be trivial (because with trivial monodromy), so $\alpha=\gamma=0$;
because it is trace-free, we also get $\delta=0$.
Then, denoting $\beta=dz$ for the uniformizing variable $z$, we see by integration that monodromy must be of the form 
$$\tilde A=\begin{pmatrix}1&c\\0&1\end{pmatrix}
\ \ \ \text{and}\ \ \ 
\tilde B=\begin{pmatrix}1&c\tau\\0&1\end{pmatrix}$$
where $(1,\tau)$ are periods of $dz$.
Therefore, $E$ is of type (2.1) or (3.2) if, and only if,
$\frac{a b_1}{b}=\tau$. Then to decide between (3.2) or (2.1),
we just have to check if $\frac{b}{a^\tau}\subset\Jac(X)$
is $3$-torsion or not.

The remaining cases are similar.
\end{proof}

\subsection{A universal family}\label{sec:universalfamilyTu}

Let us revisit Tu's moduli space. For any $(b_1,b_2)\in(\mathbb C^*)^2$, let us consider the representation given by 
\begin{equation}\label{eq:universalfamilydecomposable}
A=\begin{pmatrix}1&0&0\\0&1&0\\0&0&1\end{pmatrix}\ \ \ \text{and}\ \ \ 
B=\begin{pmatrix}b_1&0&0\\0&b_2&0\\0&0&\frac{1}{b_1b_2}\end{pmatrix}
\end{equation}
and the corresponding family of decomposable flat vector bundles
$E_{(A,B)}=L_1\oplus L_2\oplus L_3$. We have a natural map
$$(\mathbb C^*)^2\to \mathrm{Bun}_{\mc{O}_X}(3,X)\cong \check{\PP}^2\ ;\ (b_1,b_2)\mapsto [E_{(A,B)}].$$
By section \ref{sec:monodromyofflatbundles}, the line bundle
$L_1$ is determined by $b_1\in\mathbb{C}^*/<e^{2i\Pi\tau}>\cong\Jac(X)$. We get a natural covering map:
$$(\mathbb C^*)^2\to (\Jac(X))^2\ ;\ (b_1,b_2)\mapsto (L_1,L_2)$$
through which the previous map factors: we get a $6$-fold ramified covering
$$(\Jac(X))^2\to \mathrm{Bun}_{\mc{O}_X}(3,X)\cong \check{\PP}^{2}\ ;\ (L_1,L_2)\mapsto [L_1\oplus L_2\oplus(L_1\otimes L_2)^{-1}]. $$
This covering is the quotient of $(\Jac(X))^2$
by the symmetric group $S_3$ acting by permutation of factors
of $L_1\oplus L_2\oplus L_3$.
It is ramifying over the dual cuspidal curve
$\check{X}\subset\check{\PP}^2$. 
So the family $E_{(A,B)}$ is a universal family for the moduli space $\mathrm{Bun}_{\mc{O}_X}(3,X)$. 

We can locally describe the ramified covering near the trivial bundle $E=\mathcal O_X\oplus\mathcal O_X\oplus\mathcal O_X$ as follows. Local coordinates of $(\Jac(X))^2$ at $(\mathcal O_X,\mathcal O_X)$ are given by $(z_1,z_2)$ where $b_j=e^{2i\pi z_j}$.
Then, the permutation of factors induces the transformation group
generated by:
\begin{equation}\label{eq:permutations}
\sigma_{12}(z_1,z_2)=(z_2,z_1),\ \ \ \sigma_{23}(z_1,z_2)=(z_1,-z_1-z_2)\ \ \ \text{and}\ \ \ \sigma_{13}(z_1,z_2)=(-z_1-z_2,z_2).
\end{equation}
The invariant polynomial functions are generated by:
$$F_2(z_1,z_2)=z_1^2+z_1z_2+z_2^2\ \ \ \text{and}\ \ \ 
F_3(z_1,z_2)=z_1z_2(z_1+z_2).$$
The quotient map is locally given by 
$$(z_1,z_2)\mapsto (F_2,F_3)$$
sending fixed points of $\sigma_{ij}$
$$\Delta_{12}:\{z_1=z_2\},\ \ \ \Delta_{23}:\{z_1+2z_2=0\}\ \ \ \text{and}\ \ \ \Delta_{13}:\{2z_1+z_2=0\}$$
onto the cusp 
$$\left(\frac{F_2}{3}\right)^3=\left(\frac{F_3}{2}\right)^2.$$

As we shall see (see Propositions \ref{prop:stableunstable}, \ref{prop:P1fiber(1)}), those vector bundles admitting stable parabolic structures are exactly those of type (1), (2.1) and (3.1).
We see that each $S$-equivalence class in Tu's moduli space
contains exactly one isomorphism class of vector bundle of that type. We can provide an alternate universal family as follows.
For $(b_1,b_2)\in(\mathbb C^*)^2$, let us now consider the representation given by 
\begin{equation}\label{eq:universalfamily}
A=\begin{pmatrix}1&0&0\\0&1&0\\0&0&1\end{pmatrix}\ \ \ \text{and}\ \ \ 
B=\begin{pmatrix}b_1&1&0\\0&b_2&1\\0&0&\frac{1}{b_1b_2}\end{pmatrix}
\end{equation}
and the corresponding family of flat vector bundles
$E_{(A,B)}$. 

\begin{prop}\label{prop:universalparabolicbundle}
The family of vector bundles $(b_1,b_2)\mapsto E_{(A,B)}$ given by (\ref{eq:universalfamily}) only contains vector bundles of types
(1), (2.1) and (3.1) near $(b_1,b_2)=(1,1)$.
\end{prop}

We thus get a local universal family at an analytic neighborhood of the trivial bundle (which is one of the cusps of $\check{X}$).

 \begin{proof}
     For instance, consider $b_1=b_2\not=1$. Then we have two distinct eigenvalues $b=b_1$, and $b^{-2}$, and we can conjugate the pair $(A,B)$ to 
     $$
A'=\begin{pmatrix}1&0&0\\0&1&0\\0&0&1\end{pmatrix}\ \ \ \text{and}\ \ \ 
B'=\begin{pmatrix}b&1&0\\0&b&0\\0&0&b^{-2}\end{pmatrix}
$$
and we deduce that $E_{(A,B)}\cong E_{(A',B')}$ is of type (2.1).
Note that replacing by $(b_1,b_2)=(b,e^{2i\pi\tau}b)$
will give a vector bundle in the same $S$-equivalence class as $(b,b)$, but will be of the type
(2.2), i.e. decomposable.
 \end{proof}


\section{The moduli space of rank 3 trivial determinant, parabolic bundles}

In this section we recall the definition of parabolic (semi)stability and show that in our set-up it implies usual vector bundle (semi)stability. We describe the wall-chamber decomposition in our case and then classify the bundles that are always stable, never stable and those that are stable only with respect to the weights in one of the chambers.

\subsection {Parabolic vector bundles}
We now recall some well-known definitions and results about parabolic vector bundles. Although the definitions and results exist in higher rank and for more points, to make it easier for the reader we only state them in our set-up. For the more general version we refer the reader to \cites{M-S, M-Y}.

\begin{defi}:\label{nota-parbundle}
A \emph{parabolic bundle} $E$ of rank $3$ on X is a vector bundle $E$ of rank $3$ on $X$ with a \emph{parabolic structure} on $p$ given by
\begin{enumerate}
    \item a filtration of the fibre $E_{p}$ (which is a vector space of dimension $3$)
    \[E|_{p} = E_{p,3} \supsetneq E_{p,2} \supsetneq E_{p,1} \supsetneq {0}\]
    \item a sequence of real numbers $\bmu = (\mu_1, \mu_2, \mu_3) \in \mathbb{R}^3$ called \emph{parabolic weights} satisfying
\begin{itemize}
    \item $\mu_3 \leq  \mu_2 \leq  \mu_1$;
     \item  $\mu_1 - \mu_3 < 1$; and
    \item $\sum\limits_{i=1}^{3}\mu_{i} =0.$
\end{itemize}
\end{enumerate}
The parabolic structure is said to have \emph{full flags} since
$\mathrm{dim}(E_{p,i}/E_{p,i+1}) = 1$ for all $i \in \{1,2\}$. We will only be considering full flags.
\end{defi}

\begin{rem}\label{range of weights}
In the above definition we do not follow the usual convention in the literature to choose the weights $\mu_i$ to be between $0$ and $1$. However, as we see later in Remark \ref{lemma-parsemistability} this convention is not necessary as adding a constant to the weights does not change the stability condition.
\end{rem}

\begin{defi} 
The \emph{parabolic degree} of a parabolic bundle $(E,p)$ is defined as
\[ \mathrm{pardeg}(E,\Eb) := d + \sum\limits_{i =1}^{3}\mu_{i}, \]
where $d$ denotes the degree of the vector bundle. 
The \emph{parabolic slope} of $(E,\Eb)$ is defined as: 
\[\mu_{\mathrm{par}}(E,\Eb) := \frac{\mathrm{pardeg}(E,\Eb)}{\rk(E)}.\]
\end{defi}

\begin{defi}\label{defi: parabolic subbundle} 
A \emph{parabolic subbundle} $(E',\Eb')$ of $(E,\Eb)$ is a subbundle $E'\subset E$ of rank $r'$ with an \emph{induced parabolic structure}. The \emph{induced parabolic structure} on $E'$ at the point $p$ is given as follows:
\begin{enumerate}
\item the quasi-parabolic structure on $E'$ at the point $p$, i.e. the filtration in $E'_p$ is given by
\[ E'_{\bullet}: E'_{p} = E_{p,r'} \supsetneq E'_{p,r'-1} \supsetneq   E'_{p,1} \supsetneq {0},\]
where $E'_{p,i} = E'_{p} \cap E_{p,i}$, i.e. we are considering the intersection with the already given filtration in $E_{p}$, and also scrapping all the repetitions of subspaces in the filtration.
\item The parabolic weights ${\mu'_{r'}} < \dots < {\mu'_{1}}$
 are taken to be the largest possible among the given parabolic weights which are allowed after the intersections, i.e.
 \[ \mu'_{i} = \mathrm{min}_{j}\{\mu_{j}|E'_{p}\cap E_{p,j} = E'_{p,i}\} = \mathrm{min}_{j} \{ \mu_j |E'_{p,i}\subseteq E_{p,j} \}\]
That is to say, the parabolic weight associated to $E'_{p,i}$ is the weight $\mu_j$ such that $E'_{p,i} \subseteq E_{p,j}$ but $E'_{p,i}\nsubseteq E_{p,j-1}$.
\end{enumerate}
 \end{defi}

\begin{rem}\label{rank of sub-bundles}
 Note that in our set-up $(E,p)$ is of rank $3$, therefore the rank $r'$ of the parabolic subbundle $(E',p)$ in Definition \ref{nota-parbundle} can only be $1$ or $2$. 
 \end{rem}

\begin{defi}\label{defi:parstab}
A parabolic bundle $(E,\Eb)$ is \emph{parabolic (semi)stable} if 
for any nonzero proper parabolic subbundle $(E',\Eb') \subset (E,\Eb)$, we have $\mu_{\mathrm{par}}(E',\Eb') (\leq) <\mu_{\mathrm{par}}(E,\Eb)$.
In other words, a parabolic vector bundle $(E,\Eb)$ is (semi)stable iff for any parabolic subvector bundle $(E',p)$: 
\[ \pind((E,\Eb), (E',\Eb')):= \mu_{\para}(E,\Eb) - \mu_{\para}(E',\Eb') \geq 0 \]
with equality holding for the strictly semistable parabolic vector bundles $(E,\Eb)$. 
\end{defi}

\begin{rem} \label{lemma-parsemistability}
Since $\sum_{i=1}^{3} \mu_i=0$, it is easy to see that $(E,\Eb)$  is parabolic (semi)stable if for any subbundle $F \subseteq E $, $\mathrm{pardeg}(F,{F_{\bullet}}) (\leq)<0$. Note that although in the usual definition of parabolic bundles, one does not ask for $\sum_{i=1}^{3} \mu_i=0$, we can assume this because adding the same constant to each  parabolic weight $\mu_i$  does not affect the (semi)stability of a parabolic vector bundle  given by a full flag. Indeed, suppose $(E,\Eb)$ is parabolic (semi)stable for the parabolic weights $(\mu_1, \mu_2, \mu_3) \in \R^3$ and $\Eb$ is given by full flag.  
Let $\lambda \in \R$, we claim that $(E,\Eb)$ is still parabolic (semi)stable for the parabolic weights $(\mu_1 + \lambda, \mu_2 + \lambda, \mu_3 + \lambda) \in \R^3$.  Indeed, let $(E',\Eb ')$ be a parabolic subbundle of $(E,\Eb)$. Then
\[ \frac{\deg E + \sum\limits_{i=1}^{3} (\mu_i + \lambda)}{3} -  
\frac{\deg E' + \sum\limits_{j \in J}^{} (\mu_j + \lambda)}{\mathrm{rk} E'} \\ 
= \frac{\deg E + \sum\limits_{i=1}^{3} \mu_i}{3} -  
\frac{\deg E' + \sum\limits_{j \in J}^{} \mu_j }{\mathrm{rk} E'}  
\]
where $J \subseteq \{1, 2, 3\}$ is the index set for the parabolic structure induced in $E'$ and the equality follows from the fact that $|J| = \rk E'$. Therefore, $\pind((E,\Eb), (E',\Eb'))$ does not change if we change $\mu_i \mapsto \mu_i + \lambda$, $i=1, 2, 3$.

Therefore given any $d \in \R$, we can translate $(\mu_1, \mu_2, \mu_3) \in \R^3$ by $\lambda$ so that $\sum\limits_{i=1}^3 \mu_i = d$. Setting $\lambda = - (\sum_{i=1}^{3} \mu_i)/3$ gives $d = 0$ as required in Definition \ref{nota-parbundle}. 
\end{rem}

\begin{defi}
 We denote by $\Bun^{{\bmu}}_{\mc{O}_X}(3,X,p)$ the moduli space of ${\bmu}$-semistable, rank $3$ parabolic bundles on the curve $X$ with trivial determinant, on the curve $X$. In \cite{M-S}, Mehta and Seshadri showed that this is a normal, projective variety. When ${\bmu}$ is \emph{generic}, i.e when parabolic stability and semistability coincide, this variety is in fact smooth. 
\end{defi}

In general, the dimension of the moduli space of ${\bmu}$-semistable parabolic bundles (even for $g=0,1$, see \cite{Boden-Yokogawa}) is given by $(r^{2}-1)(g-1) + \sum \limits_{p \in D}f_p$ where $r$ denotes the rank, $g$ is the genus of the curve, $D$ is the set of parabolic points and $f_p$ denotes the dimension of the flag. Recall that the dimension of the flag $f_p$ is $\frac{1}{2}(r^{2} - \sum\limits_{i=1}^{s_p}m_{i,p})^2$, where $m_{i,p} = \dim E_{i+1,p}- E_{i,p}$ and $s_p$ is the length of the flag. Since we consider a unique parabolic point $p$, complete flags and rank $3$, for us $s_p=3, m_{i,p} =1, r =3, g=1$, we have:

 \[\dim(\mathrm{Bun}^{{\bmu}}_{\mc{O}_X}(3,X,p)) = (9-1)(1-1) + \frac{1}{2}(9-3) = 3.\]


\subsection{Parabolic semistability and usual vector bundle semistability}\label{sec:StabilityParabVB}  
We now discuss the connection between usual semistability and parabolic $\bmu$-semistability. Note that in general, the usual slope semistability of a vector bundle (i.e. if $d/r \geq d'/r'$, for $d',r'$ the degree and rank of any proper subvector bundle) does not imply \emph{parabolic slope semistablity} or vice-versa (see \cite{NFV} for when a bundle is $\bmu$-semistable but not semistable as a vector bundle).  However, in our set-up as we show below parabolic semi-stability does imply the usual slope semistability, (but not vice-versa).  Therefore, the only semistable parabolic bundles in our set-up are those stated in Theorem \ref{Atiyah:semistablebundles}.


\begin{prop} \label{prop-psemistableimpliessemistable}
Let $(E,\Eb)$ be a $\bmu$-semistable bundle of rank $3$, trivial determinant with parabolic structure given at a unique point $p$. Then $E$ is semistable as a vector bundle. 
\end{prop}

\begin{proof} Suppose there exists a subbundle $F \subseteq E$ such that $\mu(F) > \mu(E)$. Thus, 
    $\deg(F)/\rk(F) \geq 1.$
Since we are assuming $(E,\Eb)$ to be $\bmu$-semistable, by Remark \ref{lemma-parsemistability} for any $E' \subseteq E$, $\mu_{par}(E') \leq0$. Moreover, given the assumption over the parabolic weights, as remarked before, $|\mu_i| < 1/2$ for all $i=1,2,3$. Hence, if $\rk(F)=1$,
\[ \mu_{par}(F) = \deg(F) + \mu_i \geq 1 + \mu_i > 0, \]
 for some $i \in \{1,2,3\}$. If  $\rk(F)=2$, 
\[ \mu_{par}(F) =  \frac{\deg(F)}{\rk(F)}+ \mu_i + \mu_j \geq 1 + \mu_i+ \mu_j = 1 - \mu_k > 0, \]
where $\{i,j,k\} = \{1,2,3\}$. In both cases we have a contradiction, and the proposition follows. 
\end{proof}


\begin{cor} \label{prop-subbundlesofged0}
Let $(E, \Eb)$ be a rank $3$ parabolic vector bundle with trivial determinant bundle. Let $(E', E'_{\bullet})$ be a parabolic subbundle of $(E, \Eb)$. If $\deg E' < 0$, then $E'$ does not parabolic destabilize $(E,\Eb)$, regardless of the flag $E'_{\bullet}$. In particular, to check parabolic semistability of $E$, we only have to consider subbundles of degree $0$.  
\end{cor}

\begin{proof}
    Let $(E,p)$ be $\bmu$-semistable.  By Proposition \ref{prop-psemistableimpliessemistable} $\deg E' \leq 0$ for all subbundle $E' \subseteq E.$ Suppose $\deg E' < 0$, since $|\mu_i| < 1/2$, then $\mu_{par}(E',p)<0.$ Therefore, $E'$ does not parabolic destabilize $E$. 
\end{proof}


\begin{prop} \label{prop-pslopesopfsubbundles}
    Let $(E,\Eb)$ be as Definition \ref{nota-parbundle}. Let $L \subset E$ (resp. $F \subset E$ ) runs over the rank $1$ (resp. rank $2$) subbundles of $E$ of degree $0$. Then the following holds. 
\begin{enumerate}
    \item If $L_p \not\subseteq E_{p,2}$ and $E_{p,1} \not\subseteq F_p$, then $(E,\Eb)$ is stable for any choice of $\mu_2 \neq 0$. 

    \item If $L_p \not\subseteq E_{p,2}$ and $E_{p,1} \subseteq F_p \neq E_{p,2}$, then $(E,\Eb)$ is stable if, and only if, $\mu_2 >0$. Moreover, it is strictly semistable if, and only if, $\mu_2=0$. 

    \item If $L_p \neq E_{p,1}$, $L_p \subseteq E_{p,2}$ and $E_{p,1} \not\subseteq F_p$, then $(E,\Eb)$ is stable if, and only if, $\mu_2 < 0$. Moreover, it is strictly semistable if, and only if, $\mu_2 =0$. 

    \item If $L_p \neq E_{p,1}$, $L_p \subseteq E_{p,2}$ and $E_{p,1} \subseteq F_p \neq E_{p,2}$, then $(E,\Eb)$ is strictly semistable if, and only if, $\mu_2 =0$. Moreover, it is unstable for $\mu_2 \neq 0$. 

    \item If $L_p = E_{p,1}$ or $F_p = E_{p,2}$, then $(E,\Eb)$ is unstable for every $(\mu_1, \mu_2, \mu_3) \neq (0,0,0)$. 
\end{enumerate} 
\end{prop}

\begin{proof}
 In what follows, let $L, F \subseteq E$, where $L$ (resp. $F$) runs over the rank $1$ (resp. rank $2$) subbundles of $E$ of degree $0$. 

    If $L_p = E_{p,1}$, then the induced flag at $L$ is  $L_{\bullet}: 0 \subseteq L_p = L_p \cap E_{p,1}$ and thus
    $\mathrm{pardeg}(L,L_{\bullet}) = \mu_1$. 
    If $L_p \neq E_{p,1}$ and $L_p \subseteq E_{p,2}$, then the induced flag at $L$ is $L_{\bullet}: 0 \subseteq L_p = L_p \cap E_{p,2}$  and thus $\mathrm{pardeg}(L,L_{\bullet}) = \mu_2$. 
    If $L_p \not\subseteq E_{p,2}$, then the induced flag at $L$ is  $L_{\bullet}: 0 \subseteq L_p = L_p \cap E_{p,3}$ and thus
    $\mathrm{pardeg}(L, L_{\bullet}) = \mu_3$. 

    If $E_{p,1} \subseteq F_p$ and $F_p = E_{p,2}$, then the induced flag at $F$ is $F_{\bullet}: 0 \subseteq E_{p,1} \subseteq F_p = F_p \cap E_{p,2}$ and thus
    $\mathrm{pardeg}(F,F_{\bullet}) = \mu_1 + \mu_2$. 
    If $E_{p,1} \subseteq F_p$ and $F_p \neq E_{p,2}$, then the induced flag at $F$ is $F_{\bullet}: 0 \subseteq E_{p,1} \subseteq F_p = F_p \cap E_{p,3}$ and thus
    $\mu_{par}(F,F_{\bullet}) = \mu_1 + \mu_3$. 
    If $F_p \cap E_{p,2}$ is a line different from $E_{p,1}$, then the induced flag at $F$ is $F_{\bullet}: 0 \subseteq F_p \cap E_{p,2} \subseteq F_p = F_p \cap E_{p,3}$ and thus
    $\mu_{par}(F,F_{\bullet}) = \mu_2 + \mu_3$. 

  Therefore, the proof follows from Remark \ref{lemma-parsemistability}, Corollary \ref{prop-subbundlesofged0} and putting all the above calculations together. \end{proof}


\subsection{The Wall Chamber decomposition of the weight chamber}\label{sec:WallChamberDecomposition}

In the following, we recall the wall-chamber decomposition of admissible weights for rank $3$ parabolic vector bundles with trivial determinant and parabolic structure given at a unique point.

By Definition \ref{nota-parbundle}, and as observed in Remark \ref{lemma-parsemistability}, in our situation, the parabolic weights $\bmu = (\mu_1, \mu_2, \mu_3) \in \mathbb{R}^3$ must satisfy the following conditions: 
\begin{itemize}
    \item $\mu_1 \geq  \mu_2 \geq  \mu_3$;
    \item $\mu_1 - \mu_3 < 1$; and
    \item $\sum\limits_{i=1}^{3}\mu_{i} =0.$
\end{itemize}
The above conditions cover the triangle $\triangle \subset [-1,1]^{3}$ with  vertices $(0,0,0)$, $(\frac{2}{3},\frac{-1}{3},\frac{-1}{3})$, $(\frac{1}{3},\frac{1}{3},\frac{-2}{3})$. 

By Proposition \ref{prop-pslopesopfsubbundles}, there are essentially two stability conditions, given by $\mu_{2}>0$ and $\mu_2<0$ which means the region $\triangle$ can be split into two chambers, say $\Pplus$, $\Pmin$ (for $\mu_{2}>0$, $\mu_2<0$ resp.) with the wall being given by $\mu_2=0$.
We represent the space of admissible weights with the wall-chamber decomposition just introduced, in Figure \ref{fig-admissibleweights}. We note that this is also given in \cite [Example $2.9$] {S-T}.

\begin{figure}[h!]
\centering
\begin{tikzpicture}[scale=3]

\coordinate (O) at (0,0);
\coordinate (A) at (-1,1);
\coordinate (B) at (1,1);

\draw[thick] (O) -- (A);
\draw[thick] (O) -- (B);
\draw[thick] (A) -- (B);
\draw[thick, red] (O) -- (0,1);

\node[below]  at (O) {$(0,0,0)$};
\node[left] at (A) {$\left(\frac23,-\frac13,-\frac13\right)$};
\node[right] at (B) {$\left(\frac13,\frac13,-\frac23\right)$};

\node[above] at (0,1) {$\mu_2=0$};
\node[right] at (-0.5,0.7) {$\Pmin$};
\node[left] at (0.5,0.7) {$\Pplus$};

\end{tikzpicture}
\caption{Admissible parabolic weights for trivial determinant, rank $3$ bundles and a unique parabolic point}
\label{fig-admissibleweights}
\end{figure}
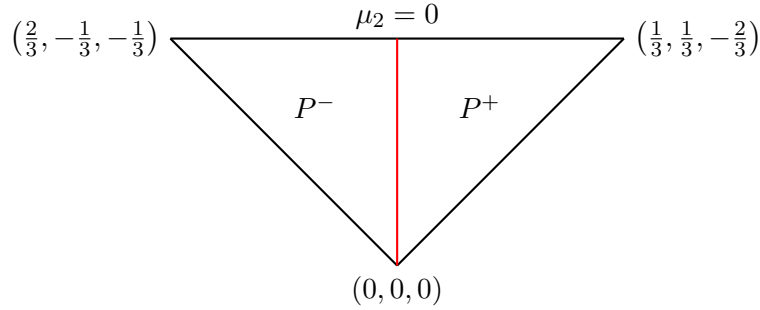    

\begin{rem}
When $\bmu\in \Pplus$ or $\bmu\in \Pmin$, every $\bmu$-semistable bundle is $\bmu$-stable, and the set of $\bmu$-stable bundles only depends on the chamber, and not of the specific choice of $\bmu$ inside.
\end{rem}

\subsection{The always stable parabolic bundles and the never stable ones}\label{sec:StableParabAllChambers} 

Since parabolic stability depends on the choice of the flag, a certain choice of flag ensures that the parabolic bundle is always parabolic stable. 

\begin{prop}\label{prop:stableunstable} We have the following conditions for \emph{always stable}  and \emph{never stable} bundles:
\begin{enumerate}
\item A parabolic vector bundle $(E,\Eb)$ is $\bmu$-stable for either of the  chambers $\Pmin$
and $\Pplus$ if, and only if, it satisfies the two conditions:
\begin{enumerate}

\item $L_p \not\subseteq E_{p,2}$ for every line subbundle $L \subseteq E$ of degree $0$; 
\item $E_{p,1} \not\subseteq F_p$
for every rank $2$ subbundle $F \subseteq E$ of degree $0$.
\end{enumerate}
In other words, the parabolic bundles $(E,\Eb)$ satisfying these conditions are always stable, and we  will denote the set of such parabolic bundles by $\Ugen$. 
\item Vector bundles of types (2.2), (3.2) and (3.3) do not admit
    $\bmu$-semistable parabolic structure: they are $\bmu$-unstable
    for the two chambers $\Pmin$ and $\Pplus$.
    \end{enumerate}
\end{prop}

\begin{proof} Part 1 easily follows from  Proposition \ref{prop-pslopesopfsubbundles}.

\noindent For part $2$, one easily checks, by using the description of subbundles of degree $0$ in Proposition \ref{prop:SubbundlesOfDegree0},
    that for each vector bundle $E$ of the statement, and each 
    parabolic structure $\Eb$, we can find subbundles
    $L,F\subset E$ of degree $0$ such that:
    \begin{enumerate}
\item $L_p \subset E_{p,2}$; or
\item $F_p \supset E_{p,1}$.
\end{enumerate}
For instance, in the case $E$ has type (2.2), then we see from 
Proposition \ref{prop:SubbundlesOfDegree0} that there are $1$-parameter families of rank $1$ and rank $2$ subbundles;
reasoning in restriction to a fiber by means of 
Corollary \ref{cor:injectiverestriction}, one easily check that we can find $L$ and $F$ as above.
\end{proof}

\subsection{Semistable parabolic bundles corresponding to each chamber}\label{sec:StableParabSingleChamber}

We now describe the parabolic bundles $(E,\Eb)$ that are $\bmu$-stable only with respect to the weights $\bmu$ in one of the chambers $\Pplus$ and $\Pmin$.

\begin{thm}\label{thm:decompositionUgensigma} \label{thm:classifying stab parabolic bundles} 
Let $\Mmin$ (resp. $\Mplus$) be the moduli space of $\bmu$-semistable bundles with trivial determinant for ${\bmu} \in \Pmin$ (resp. ${\bmu} \in \Pplus$). Let $\Ugen$ be as in Proposition \ref{prop:stableunstable}. Then 
 
\[ \Mmin = \Smin \cup \Ugen \quad \text{and} \quad  \Mplus = \Splus \cup \Ugen,\]
where: 
\[ \Smin := \{ (E, \Eb) \;|\; E_{p,1} \not\subseteq L_{p,1} \subseteq E_{p,2} \text{ and } E_{p,1} \nsubseteq F_{p,1} \} \] 
and 
\[ \Splus := \{ (E, \Eb) \;|\;  L_{p,1} \nsubseteq E_{p,2} \text{ and } E_{p,1} \subseteq F_{p,2} \neq E_2 \}, \] 
$L \subset E$ (resp. $F \subset E$ ) runs over the rank $1$ (resp. rank $2$) subbundles of $E$ of degree $0$ and $E$ is one of the  three types of vector bundles: (1), (2.1), (3.1).
\end{thm}

\begin{proof} Let $(E,\Eb)$ as Notation \ref{nota-parbundle}. Since  $(E,\Eb)$ is $\bmu$-semistable, Proposition  \ref{prop-psemistableimpliessemistable} implies $E$ be semistable. Thus Atiyah's classification of vector bundles over elliptic curves, i.e.\ Theorem \ref{Atiyah:semistablebundles}, together with Proposition \ref{prop:stableunstable} yields that $E$ must be as stated. Next, follows from Proposition \ref{prop-subbundlesofged0} that to check parabolic semistability of $(E,\Eb)$, we just need to consider subbundles of degree $0$. The theorem thus follows from Proposition \ref{prop-pslopesopfsubbundles}.  
\end{proof}

\section{Geometry of the moduli space of rank 3 stable parabolic bundles with trivial determinant }

 In this section we describe the geometry of the moduli spaces $\Bun^{\bmu}_{\mc{O}_X}(3,X,p) = M^{\pm}$. We show that $M^{\pm}$ is a $\PP^1$-bundle over $\Bun_{\mc{O}_X}(3,X) = \PP^2$ and $\check{\PP}$. We construct a local universal family of parabolic bundles and show that there exists an isomorphism between the two compactifying loci.

\subsection{The moduli space of semistable parabolic bundles is a ruled surface}

We have a natural proper morphism 
$$\Pi^{\bmu} : \Bun^{\bmu}_{\mc{O}_X}(3,X,p)\to \Bun_{\mc{O}_X}(3,X)$$
induced by Tu's modular map, and show that all fibers are isomorphic to $\PP^1$. Then a classical result of Fischer-Grauert \cite{Fischer-Grauert-65}, the morphism $\Pi$ makes our moduli space into a locally trivial holomorphic $\PP^1$-bundle.

For any $S$-equivalence class in $\Bun_{\mc{O}_X}(3,X)$,
only one isomorphism class of vector bundle can be equipped by a $\bmu$-semistable parabolic structure, namely a vector bundle $E$ of type (1), (2.1) or (3.1). So the fiber of $\Pi^{-1}([E])$
is given by the moduli space of $\bmu$-stable parabolic structures $\Eb$ on $E$, up to automorphism of $E$.
We now discuss all possible cases.

\begin{prop}\label{prop:P1fiber(1)}\it
    Let $E$ be a vector bundle of respective types (1), (2.1) and (3.1). 
    Then the moduli space $\Pi^{-1}([E])$ of $\bmu$-stable parabolic structures $\Eb$ on $E$,
    up to automorphism, is $\PP^1$, 
    for either ${\bmu} \in \Pmin$, or ${\bmu} \in \Pplus$. Moreover, this fiber $\Pi^{-1}([E])\cong\PP^1$ intersects
    $\Smin$ (resp. $\Splus$) in exactly $3$, $2$ and $1$ points, respectively.
\end{prop}

\begin{proof}
    Recall that subbundles of degree $0$ and automorphisms are determined by their restriction to the fiber $E_p$: they are represented by points and line in $\PP(E_p)$.
    
    {\bf Case $E$ of type (1)}. In $\PP(E_p)$, we can choose the factors $L_i$ of $E$ to be normalized in homogeneous $[Z_1,Z_2,Z_3]$ coordinates as:
    $$L_1=[1,0,0],\ \ \ L_2=[0,1,0],\ \ \ L_3=[0,0,1];$$
    any other degree $0$ subbundle is one of the lines $\mathcal L_{ij}$ defined by $L_i\oplus L_j$. We denote by $P\in \mathcal L$ the point and line defined by the parabolic structure $E_{p,1}\subset E_{p,2}$.

    Assume ${\bmu} \in \Pmin$ and $\Eb$ be a ${\bmu}$-semistable parabolic structure. Then $E_{p,1}\not\in L_i\oplus L_j$ for any $i,j=1,2,3$, meaning $P\not\in\mathcal L_{ij}$.
    The automorphisms of $E$ take the form 
    $$[Z_1,Z_2,Z_3]\mapsto [a_1 Z_1,a_2 Z_2,a_3 Z_3]$$
    and are acting uniquely transitively on the complement of lines $\mathcal L_{ij}$: we can assume $P=[1,1,1]$ and we have fixed automorphism freedom. Therefore, the parabolic structure is determined by the line $\mathcal L$ 
    passing through $[1,1,1]$. The set of those lines is parametrized by $\PP^1$, and we can check that any such line gives rise to a ${\bmu}$-semistable parabolic structure. We note that the parabolic structure belongs to $\Smin$ if, and only if, $\mathcal L\ni L_i$, $i=1,2,3$, which gives $3$ distinct points in the moduli space.

    Assume ${\bmu} \in \Pplus$ and $\Eb$ be a ${\bmu}$-semistable parabolic structure. Then $\mathcal L\not\ni L_i$ for any $i=1,2,3$.
    The automorphisms group acts uniquely transitively 
    on the lines $\mathcal L$ with this property, 
    and we can assume $\mathcal L:\{Z_1+Z_2=Z_3\}$ and the automorphism freedom is fixed. Then the parabolic structure is given by the position of $P$ along $\mathcal L\cong\PP^1$. The $3$ points $\mathcal L\cap\mathcal L_{ij}$, $i,j=1,2,3$, define the points of $\Splus$ in the fiber.

    In Figure \ref{fig:normal}, we can see the two normalizations used in the proof of Proposition \ref{prop:P1fiber(1)}:
\begin{figure}[H]
    \centering
    \includegraphics[width=0.6\linewidth]{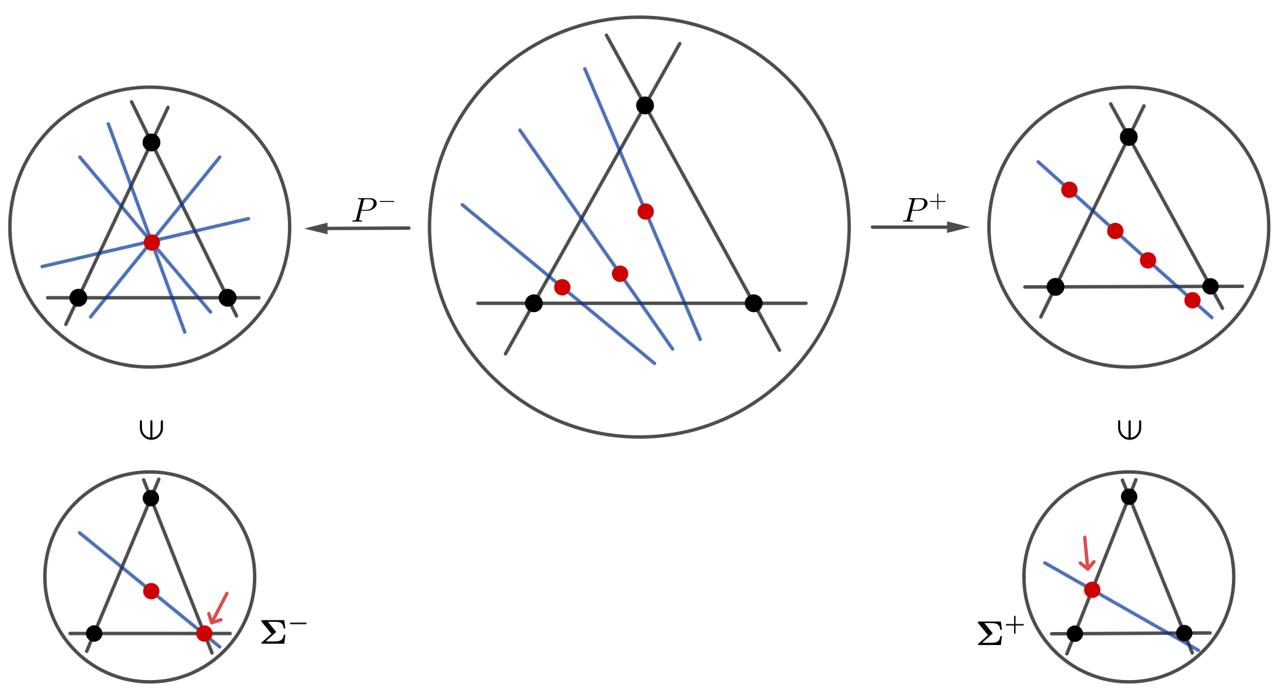}
    \caption{Normalizations of parabolic structures}
    \label{fig:normal}
\end{figure}

    {\bf Case $E$ of type (2.1)}: $E=L^{-2}\oplus(L\otimes\mathcal E_2)$ with $L^3\not=\mathcal O_X$.
     In $\PP(E_p)$, we can choose homogeneous coordinates $[Z_1,Z_2,Z_3]$ so that:
     \begin{itemize}
         \item $L^{-2}$ is the point $P_2=[0:0:1]$,
         \item $L\otimes\mathcal E_2$ is the line $\mathcal L_1:\{Z_3=0\}$,
         \item $L\subset(L\otimes\mathcal E_2)$ is the point $P_1=[1:0:0]$,
         \item $L^{-2}\oplus L$ is the line $\mathcal L_2:\{Z_2=0\}$.
     \end{itemize}
     We denote by $P\in \mathcal L$ the point and line defined by the parabolic structure $E_{p,1}\subset E_{p,2}$.

     Assume ${\bmu} \in \Pmin$ and $\Eb$ be a 
     ${\bmu}$-semistable parabolic structure. Then $P\not\in \mathcal L_1\cup\mathcal L_2$.
    The automorphisms of $E$ take the form 
    $$[Z_1,Z_2,Z_3]\mapsto [a Z_1+bZ_2,a Z_2,c Z_3]$$
    and are acting uniquely transitively on the complement of lines $\mathcal L_1\cup\mathcal L_2$ so that we can assume $P=[1,1,1]$ and 
    conclude, as in the preceding proof, that the set of parabolic structures is given by the set of lines $\mathcal L$ passing through $P$, parametrized by $\PP^1$. Moreover, $\Smin$ intersects
    at the two points corresponding to the lines $\mathcal L$ passing through $P_1$ and $P_2$.

    Assume ${\bmu} \in \Pplus$ and $\Eb$ be a ${\bmu}$-semistable parabolic structure. Then $\mathcal L\not\ni P_1,P_2$.
    The automorphisms group acts uniquely transitively 
    on the lines $\mathcal L$ with this property, 
    and we can assume $\mathcal L:\{Z_1+Z_2=Z_3\}$ and the automorphism freedom is fixed. Then the parabolic structure is given by the position of $P$ along $\mathcal L\cong\PP^1$. The $2$ points $\mathcal L\cap\mathcal L_{i}$, $i=1,2$, define the points of $\Splus$ in the fiber.

    {\bf Case $E$ of type (3.1)}: $E=L\otimes\mathcal E_3$ with $L^3=\mathcal O_X$,
and we get a flag $L\subset L\otimes\mathcal E_2\subset E$.
     In $\PP(E_p)$, we can choose homogeneous coordinates $[Z_1,Z_2,Z_3]$ so that:
     \begin{itemize}
         \item $L$ is the point $P_1=[1:0:0]$,
         \item $L\otimes\mathcal E_2$ is the line $\mathcal L_1:\{Z_3=0\}$.
     \end{itemize}
     We denote by $P\in \mathcal L$ the point and line defined by the parabolic structure $E_{p,1}\subset E_{p,2}$.

     Assume ${\bmu} \in \Pmin$ and $\Eb$ be a 
     ${\bmu}$-semistable parabolic structure. Then $P\not\in \mathcal L_1$.
    The automorphisms of $E$ take the form 
    $$[Z_1,Z_2,Z_3]\mapsto [a Z_1+bZ_2+cZ_3,a Z_2+bZ_3,a Z_3]$$
    and are acting uniquely transitively on the complement of lines $\mathcal L_1$ so that we can assume $P=[1,1,1]$ and 
    conclude, as in the preceding proof, that the set of parabolic structures is given by the set of lines $\mathcal L$ passing through $P$, parametrized by $\PP^1$. Moreover, $\Smin$ intersects
    at the point corresponding to the line $\mathcal L$ passing through $P_1$.

    Assume ${\bmu} \in \Pplus$ and $\Eb$ be a ${\bmu}$-semistable parabolic structure. Then $\mathcal L\not\ni P_1$.
    The automorphisms group acts uniquely transitively 
    on the lines $\mathcal L$ with this property, 
    and we can assume $\mathcal L:\{Z_1+Z_2=Z_3\}$ and the automorphism freedom is fixed. Then the parabolic structure is given by the position of $P$ along $\mathcal L\cong\PP^1$. The point $\mathcal L\cap\mathcal L_1$ define the intersection of $\Splus$ with the fiber.
\end{proof}

\begin{cor}
    The moduli spaces $\Mmin$ and $\Mplus$ have the structure of a locally trivial $\PP^1$-bundle through Tu's modular map $\Pi^\pm:M^{\pm}\to\Bun_{\mc{O}_X}(3,X)\cong\PP^2$. Moreover, 
    the restriction $\Pi^\pm\vert_{\Sigma^\pm}:\Sigma^\pm\to\PP^2$
    is a degree $3$ cover possibly ramifying over the cuspidal sextic $\check{X}$.
\end{cor}

\subsection{Compactifications of the moduli spaces}

The moduli space of generic bundles $\Ugen$ is embedded
as a Zariski open subset in both $\Mmin$ and $\Mplus$ inducing a natural 
birational map $\Psi:\Mmin\dashrightarrow \Mplus$.

\begin{thm}\label{thm:MminequalMplus} The natural birational map $\Psi:\Mmin\dashrightarrow \Mplus$
is a biregular isomorphism of $\PP^1$-bundles:
\begin{diagram}
     \Mmin &\rTo^{\Psi}&\Mplus\\
     \dTo^{\Pi^-}&\circlearrowleft&\dTo_{\Pi^+}\\
      V_2 &\rTo^{\mathrm{id}}& V_2
         \end{diagram}
It induces an isomorphism between compactifying loci
$$\Psi\vert_{\Smin}:\Smin\stackrel{\sim}{\longrightarrow}\Splus.$$
\end{thm}

\begin{proof}
    The restriction of $\Psi$ to $\ugen$ is an isomorphism
    commuting with $\Pi^\pm$ as in the commutative diagram.
    In restriction to a generic fiber, it is isomorphically sending 
    $\PP^1$ minus $3$ points onto $\PP^1$ minus $3$ points;
    it automatically extends by continuity, and therefore holomorphically.
    A similar argument can be done for special fibers intersecting $\Sigma^\pm$ at $2$ or $1$ point.

    We can even be more explicit. Going back to the proof of Proposition \ref{prop:P1fiber(1)}, the fiber $(\Pi^-)^{-1}(E)$ can be parametrized by the slope $t\in\PP^1$ of the line $\mathcal L^-:\{Z_2-tZ_1=(1-t)Z_3\}$ passing through $P^-=[1:1:1]$;
    the intersection with $\Smin$ is given by $t=0,1,\infty$.
    On the other hand, the fiber $(\Pi^+)^{-1}(E)$ can be parametrized by the position of the point $P^+=[\lambda:1-\lambda:1]$
    along the fixed line $\mathcal L^+:\{Z_1+Z_2=Z_3\}$;
    the intersection with $\Splus$ is given by $\lambda=0,1,\infty$.
    One easily check that the automorphism
    $$[Z_1,Z_2,Z_3]\mapsto [(1-t) Z_1,t Z_2,t(1-t) Z_3]$$
    is sending parabolic structure $P^-\subset\mathcal L^-$
    into $P^+\subset\mathcal L^+$ with $\lambda=\frac{t}{t-1}$.
    This defines the natural isomorphism 
    $$\PP^1\setminus\{0,1,\infty\}\to \PP^1\setminus\{0,1,\infty\}\ ;\ t\mapsto \frac{t}{t-1}$$
    in restriction to $\ugen$, that extends on the whole fiber $\PP^1$ by permuting the boundary points.
\end{proof}

\subsection{A universal family for parabolic bundles}\label{sec:universalparabolic}

Before proving Theorem \ref{thm:PTP2}, let us go back to the construction of a universal family developed in section \ref{sec:universalfamilyTu}, and construct a local universal family of parabolic bundles. In order to do this, consider the family of flat vector bundles $E=E_{(A,B)}$ defined by (\ref{eq:universalfamily}). One easily check that, for $b_i\not= b_j$,  degree $0$ subbundles intersect $\PP(E\vert_p)$ at:
\begin{itemize}
    \item $L_1:[1:0:0]$,
    \item $L_2:[1:(b_2-b_1):0]$,
    \item $L_3:[(b_1b_2)^2:b_1b_2(1-b_1^2b_2):(1-b_1^2b_2)(1-b_1b_2^2)]$,
    \item $L_1\oplus L_2:\{z_3=0\}$,
    \item $L_1\oplus L_3:\{Z_2=\frac{b_1b_2}{1-b_1b_2^2}\}Z_3$,
    \item $L_2\oplus L_3:\{Z_2=(b_2-b_1)Z_1+\frac{b_1b_2}{1-b_1^2b_2}Z_3\}$.
\end{itemize}
For $(b_1,b_2)\sim(1,1)$, all three $L_i$'s are close to $[1:0:0]$.
Therefore, if we set $E_{p,2}=\{Z_1=0\}$ and $E_{p,1}=[0:w:1]$,
then $E_{p,1}\subset E_{p,2}$ defines a $\bmu$-stable parabolic structure for $\bmu\in\Pplus$. For $(b_1,b_2)\sim(1,1)$, we thus get a local universal family of parabolic vector bundles for $\Mplus$.
If we normalize the $L_i$'s and parabolic structure as in the proof
of Proposition \ref{prop:P1fiber(1)}, we get $P=[\lambda:1-\lambda:1]$ where
$$\lambda=\frac{(1-b_1b_2^2)(1-b_1^2b_2)}{b_1^2b_2^2(b_2-b_1)}w-\frac{1-b_1b_2^2}{b_1b_2(b_2-b_1)}.$$
We can consider the action induced by permutations (\ref{eq:permutations})
$$\sigma_{12}(b_1,b_2)=(b_2,b_1),\ \ \ \sigma_{23}(b_1,b_2)=(b_1,\frac{1}{b_1b_2})\ \ \ \text{and}\ \ \ \sigma_{13}(b_1,b_2)=(\frac{1}{b_1b_2},b_2).$$
on the parabolic structure, and we get 
$$\sigma_{12}:\left\{\begin{matrix}
    \lambda&\mapsto&1-\lambda\\ w&\mapsto&w
\end{matrix}\right.\ \ \ \sigma_{23}:\left\{\begin{matrix}
    \lambda&\mapsto&\frac{1}{\lambda}\\ w&\mapsto&\frac{b_1b_2w}{(b_1^2b_2-1)w+b_1b_2}
\end{matrix}\right.\ \ \ \sigma_{13}:\left\{\begin{matrix}
    \lambda&\mapsto&\frac{\lambda}{\lambda-1}\\ w&\mapsto&\frac{b_1b_2w}{(b_1b_2^2-1)w+b_1b_2}
\end{matrix}\right.$$
We thus check that permutations act trivially over their $(b_1,b_2)$-fixed point set: for instance, $\sigma_{13}$ acts trivially on $w$ over $b_1^2b_2=1$. This is why we get a $\PP^1$-bundle in the quotient.
We also have 
$$\begin{matrix}
    \lambda=0&\leftrightarrow&w=\frac{1}{b_3-b_1}\\
    \lambda=1&\leftrightarrow&w=\frac{1}{b_3-b_2}\\
    \lambda=\infty&\leftrightarrow&w=\infty\hfill
\end{matrix}$$
The points of $\Splus$ correspond to horizontal sections $\lambda=0,1,\infty$ that are permuted by the action of permutations. For instance,
$\lambda=0$ and $\lambda=1$ are permuted by the action of $\sigma_{12}$;
but they intersect transversely in $w$-coordinate over $b_1=b_2$. We deduce that $\Splus$ is smooth, branching over the smooth part of $V_1$.

\subsection{ Compactifying loci and the dual curve }

Let us consider the incidence diagram:
\begin{diagram}
     M:=&\PP(T{\PP^2}) &\rTo^{\check{\Pi}}&\check{\PP}^2&\supset\check{X}\\
     &\dTo^{\Pi}&&&\\
      X\subset&\PP^2 &&& 
         \end{diagram}
Recall that $V_2:=\Bun_{\cO_X}(3,X)$ naturally identifies with $\check{\PP}^2$, $V_1=\check{X}$ is the locus of partially undecomposable bundles,
and $V_2=\mathrm{Sing}(\check{X})$, the locus of (totally) undecomposable
bundles. Denote by $\Sigma\subset M$ the smooth subvariety defined by $\Sigma=\Pi^{-1}(X)$, a $\PP^1$-bundle over $X$.

\begin{figure}[h!]
    \centering
    \includegraphics[width=0.6\linewidth]{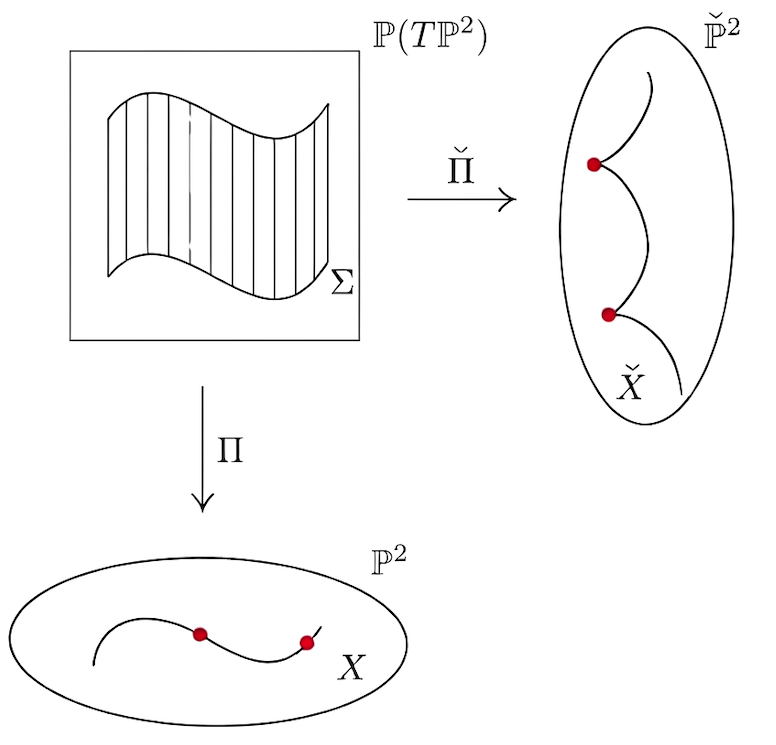}
    \caption{ The compactifying locus as a $\PP^1$-bundle over $X$}
    \label{fig:duality}
\end{figure}

On the other hand, fix a sign $\pm$, and consider the corresponding pair $(M^\pm,\Sigma^\pm)$ defined in Theorem \ref{thm:decompositionUgensigma}.

\begin{thm}\label{thm:PTP2} There is an isomorphism $\Psi^+:M\to \Mplus$
making the following diagram commutative
\begin{diagram}
     M &\rTo^{\Psi^+}&\Mplus\\
     \dTo^{\check{\Pi}}&\circlearrowleft&\dTo_{\Pi^+}\\
      \check{\PP}^2 &\rTo^{\psi}& V_2 
         \end{diagram}
and restricting to $\Sigma$ as an isomorphism
$$\Psi^+\vert_{\Sigma}:\Sigma\stackrel{\sim}{\longrightarrow}\Splus.$$
\end{thm}

\begin{proof}[Proof of Theorem \ref{thm:PTP2}]
    We first construct $\Psi$ over the main open strata $V_2^\circ=V_2\setminus V_1\subset\check{\PP}^2$. 
    In order to do this, we complete the construction of Theorem \ref{thm:TuExplicitIsomorphism} by adding the parabolic structure. A point of $\mathcal V_2^\circ:=(\check{\Pi}^+)^{-1}(V_2^\circ)$
    is a line $\mathcal L\subset\PP^2$ intersecting $X$ in $3$ distinct points $p_1,p_2,p_3$ and an additional point $p_4\in\mathcal L$.
    For the moment, we have ordered points, but we will get rid of this choice later.
    We first associate to this data the vector bundle $E=\bigoplus_{i=1}^3\cO_X(p_i-\infty)$. In order to add a parabolic structure, let us choose homogeneous coordinates $[Z_1:Z_2:Z_3]$ on the fiber $\PP(E\vert_p)$ as in the proof of Proposition \ref{prop:P1fiber(1)}.
    Then we define the parabolic structure by the data of a normalized 
    line $\mathcal L_0:\{Z_1+Z_2=Z_3\}$, and the point $P=q_4$ such that
    the following cross-ratios coincide:
    $$(p_1,p_2;p_3,p_4)=(q_1,q_2;q_3,q_4)=\lambda$$
    where $q_i=\mathcal L_0\cap \{Z_i=0\}$. 
    We observe that $p_4\in X$ correspond to $\Sigma$ while 
    $P_1,P_2,P_3$
    correspond to $\Splus$. Moreover, we can check that the parabolic bundle does not depend of the choice of the ordering of $p_1,p_2,p_3$.
    For instance, the permutation $\sigma_{12}$ permutes the two factors
    $L_1$ and $L_2$ of $E$; consequently, it permutes the points $p_1$ and $p_2$ on the one hand, and the points $P_1=[1:0:0]$ and $P_2=[0:1:0]$ on the other hand, therefore permuting $q_1$ and $q_2$.
    The two cross-ratio are both changed by 
    $\lambda\mapsto 1-\lambda$, and are still equal.
    We have constructed a birational isomorphism of $\PP^1$-bundles $\Psi^+:\Sigma\dashrightarrow M^+$ which is biregular over $V_2^\circ$, sending $\Sigma\cap\check{\Pi}^{-1}(\check{\PP}^2\setminus\check{X})$ to $\Splus\cap V_2^\circ$.

    Let us prove that $\Psi$ is also biregular over $V_1^\circ=V_1\setminus V_0$. Given a local transversal section to $z\in(\C,0)\hookrightarrow V_1^\circ$, the restriction of $M^+$ is a locally trivial $\PP^1$-bundle $(\C,0)\times\PP^1$ with coordinates $(z,w)$. The restriction of $\Splus$ has one regular sheet/section, and one double-sheet branching over $z=0$, and they do not intersect; up to change of coordinates, we may assume that $\Splus$ is given by 
    $w=\infty$ and $w^2+f(z)w+g(z)=0$ with $f,g$ holomorphic, $g(0)=0$ (Weierstrass Preparation). After a change of coordinate $w\mapsto a(z)w+b(z)$ with $a,b$ holomorphic, $a(0)\not=0$, we can assume 
    that $\Splus$ is given by 
    $w=\infty$ and $w^2=z^k$ with $k\in\Z_{>0}$ (use $b$ to kill $f$, and $a$ to normalize $g$). Then $k$ is the only local invariant of the configuration of $\Splus$ over $V_1^\circ$ (it is obviously constant on a Zariski open set $V_1'\subset V_1^\circ$). If we denote by $k+$ this invariant for $\Pi^+:(\Mplus,\Splus)\to V_2$ and $k$ the corresponding invariant for $\psi\circ\Pi:(M,\Sigma)\to V_2$,
    then $\Psi^+$ (previously constructed over $V_2^\circ$) extends as a biregular bundle isomorphims over $V_1'$ if, and only if, $k=k^+$.
    Indeed, the restriction of $\Psi^+$ to the transversal section must
    preserve $w=\infty$, and send $w^2=z^k$ to $w^2=z^{k^+}$: it must be 
    of the form $w\mapsto a(z)w$ where $a(z)=\pm z^{\frac{k^+-k}{2}}$.
    If $k\not=k^+$, then $a$ has a zero or a pole at $z=0$ (mind that
    $\Sigma$ and $\Splus$ have same monodromy, and therefore $\frac{k^+-k}{2}$ is integer); on the other hand, if $k=k^+$, then $\Psi^+$
    restricts as $w\mapsto\pm w$, and extends holomorphically.
    In our case, we know that $\Sigma$ is smooth, and we deduce that $k=1$. On the other hand, $\Splus$ can be locally constructed from the universal family of section \ref{sec:universalparabolic},
    by taking the quotient of 
    $$\{w=\infty\}\cup\{w=\frac{1}{b_3-b_1}\}\cup\{w=\frac{1}{b_3-b_2}\}$$
    by the involution $\sigma_{12}$ permuting $b_1\leftrightarrow b_2$.
    The last two factors are permuted and intersect transversely along $b_1=b_2$: this means that the invariant is $\tilde k^+=2$ before taking quotient, and is $k^+=1$ after taking quotient. In other words, starting
    with the transversal coordinate $\zeta=b_2-b_1$, we can normalize as 
    $$\{w=\infty\}\cup\{w^2=\zeta^2\}$$
    and after permutation $\sigma_{12}:(\zeta,w)\mapsto(-\zeta,w)$, we get
    $$\{w=\infty\}\cup\{w^2=z\}$$
    with $\zeta\mapsto \zeta^2=:z$ the quotient map.
    We conclude that $k=k^+=1$ and $\Psi^+$ extends biholomorphically over the neighborhood of a Zariski open set
    $v_1'\subset V_1^\circ$.

    Let us conclude that $\Psi^+$ is biregular everywhere. We note that it is biholomorphic outside of a codimension $2$ set in $V_2$.
    The biholomorphic extension is a local problem on $V_2$: 
    since the two $\PP^1$-bundles are locally trivial, we can restrict 
    to an open set $U$ where they are both trivial, and view 
    $\Psi^+$ as a birational automorphism of the trivial $\PP^1$-bundle
    on $U$, biregular outside of a codimension $2$ set $Z\subset U$.
    Let $w\in\PP^1$ be the trivializing projective coordinate.
    Maybe composing by a biregular automorphism, we can assume that 
    $w=\infty$ is not invariant, and we can write 
    $$\Psi^+:w\mapsto\frac{fw+g}{w+h}$$
    over $U\setminus Z$ with meromorphic functions $f,g,h$.
    By Levi's extension Theorem, $f,g,h$ extend as meromorphic functions on $U$ (i.e. through $Z$); after multiplying by a common denominator,
    we may rewrite
    $$\Psi^+:w\mapsto\frac{aw+b}{cw+d}$$
    with $a,b,c,d$ holomorphic without local common factor.
    But now, we have that $ad-bc$ is non vanishing over 
     $U\setminus Z$ where $\Psi^+$ is biregular; but holomorphic functions have codimension $1$ vanishing sets:
    therefore $ad-bc$ is nowhere vanishing, and $\Psi^+$ extends biholomorphically everywhere.
\end{proof}

\begin{rem}
 Note that for a generic parabolic bundle, all $E_{p,2}$ are isomorphic up to an automorphism of $E$. Therefore, the flag for $(E,p)$ is determined by choice of $E_{p,1}$. This can be seen as choosing a point on the projectivisation of $E_{p,2}$. Since $\PP(E_{p,2})$ is one-dimensional, we can thus see the parabolic structure as a $\mathbb{P}^1$. Thus the set of generic bundles $\Ugen \subset \mathrm{Bun}_{\mc{O}_X}(3,X)$ can be seen as a $\mathbb{P}^1$ bundle over $\check{\mathbb{P}}^2$. 
\end{rem}


\section{Automorphisms are modular}\label{section:Automorphisms}

The moduli spaces $(M^\pm,\Sigma^\pm)$ constructed in previous sections admit natural automorphisms arising from usual operations on parabolic bundles, that we call modular automorphisms. These operations are (see \cite{AG}):
\begin{enumerate}
    \item Taking pullback with respect to an automorphism $\sigma: X \to X$ that fixes the parabolic point $p$;
    \item tensoring by a line bundle;
    \item dualizing;
    \item applying a Hecke modification.
\end{enumerate}
Moreover, according to \cite[Lemma 5.7]{AG}, we know the composition rules of these transformations. In particular, one easily see that any finite composition of these is equivalent to the composition of at most one of each of these transformations.
Given a $3$-torsion line bundle $L\in\Jac(X)$, then the map
$$(E,E_\bullet)\mapsto (L\otimes E,L\otimes E_\bullet)$$
induces an automorphism of the moduli space $(M^\pm,\Sigma^\pm)$.
This gives an action of the $9$-group of $3$-torsion points 
$\Z/3\times\Z/3$ on $(M^\pm,\Sigma^\pm)$.
Similarly, dualization provides an involution:
$$(E,E_\bullet)\mapsto (\check{E},\check{E}_\bullet)$$
of the moduli space. 

On the other hand, Hecke modifications
shift the degree of $E$ by integers, and are determined by this integer-shift. Moreover, Hecke modification of degree $3k$ are given by tensoring with $\cO(kp)$. Since the first two operations respectively shift degree of $E$ by $3\deg(L)$, or change the sign,
the only combination of Hecke modification with other two operations that preserve the degree actually reduce to a combination of the two first operations. 

The main result of Alfaya and Gomez in \cite{AG} asserts that automorphisms of moduli spaces of parabolic bundles over curves of genus $g\ge6$ are modular. In our context, we cannot expect
such a result only considering $M^\pm$. Indeed, $M^\pm\cong\PP(T{\PP^2)}$ since any automorphism $\phi$ of $\PP^2$ can be lifted as an automorphism $\tilde\phi$ of $\PP(T{\PP^2)}$, while modular automorphisms are finite.
But if we add the special locus, then the automorphisms
of $(M^\pm,\Sigma^\pm)$ indeed are modular as well.

\begin{thm}\label{thm:ModularAutomorphism} Assume $(X,p)$ has no other automorphism than the elliptic involution $\sigma$ (fixing $p$). Then, the automorphism group of $(M^\pm,\Sigma^\pm)$ is modular, of order $18$, generated by the $9$-group of $3$-torsion points in $\Jac(X)$ together with the dualization.
Moreover, this action is just the lift to $\PP(T{\PP^2)}$ of
the automorphism group of $(\PP^2,X)$.
\end{thm}

\begin{proof} We claim that $\Phi$ must preserve or permute the
two $\PP^1$-fibrations $\Pi:M\to\PP^2$ and $\check{\Pi}:M\to\check{\PP}^2$. 
One easy argument is that the Mori cone generated by effective curves
is $$\overline{NE}_1(M)=\R_{\ge0}\{f\}+\R_{\ge0}\{\check{f}\}$$
where $f$ and $\check{f}$ represent the fibers of the two fibrations.
Then $\Phi$ is acting linearly on the Mori cone, and must preserve the boundary
defined by $\{f\}$ and $\{\check{f}\}$, namely the two fibrations
(see \cite[section 2]{Maycol} about the cohomology of $M$).
    Let $\Phi$ be an automorphism of $(M,\Sigma)$ (we omit the sign). Then $\Phi$ preserves $\Sigma$ by definition, and must preserve the unique $\PP^1$-fibration on it. By consequence, it 
    must also preserve the extended $\PP^1$-fibration 
    $\Pi:M\to\PP^2$, for instance by Blanchard's Lemma, and consequently the other one $\check{\Pi}$. In particular, it projects to automorphisms $\phi$ of $\PP^2$ and $\check{\phi}$ of $\check{\PP}^2$, and is the lift of any one of them using the contact distribution generated by the two $\PP^1$-fibrations. Since $\Sigma=\Pi^{-1}(X)$, we have that $\phi$
    is preserving the cubic $X$. By assumption, the automorphism group of $(\PP^2,X)$ has order $18$, generated by the elliptic involution, and the translation by $3$-order points.
    Indeed, flex must be permuted by $\phi$; after translation, we can assume $\pi(\infty)=\infty$ and deduce that $\phi$ is the elliptic involution. On the other hand, one can look at the action of modular automorphisms on $E$, i.e. as the action induced on lines in $\PP^2$ via Tu's automorphism in \ref{thm:TuExplicitIsomorphism}. Then we clearly see that the action of $3$-torsion line bundles 
    $$
        L_1\oplus L_2\oplus L_3\mapsto L_1'\oplus L_2'\oplus L_3',
        \ \ \  L_i'=L\otimes L_i
    $$
    induces an action by translation by $p_0$ on $X$:
    $$L_i=\cO_X(p_i-\infty)\mapsto L_i'=\cO_X(p_i+p_0-2\infty).$$
    Similarly, the dualization 
    $$L_1\oplus L_2\oplus L_3\mapsto \check{L}_1\oplus \check{L}_2\oplus \check{L}_3$$
    induces the action on $X$ by the elliptic involution
     \[L_i=\cO_X(p_i-\infty)\mapsto L_i'=\cO_X(\sigma(p_i)-\infty),\]
     which completes the proof. 
\end{proof}


\section{A higher rank Torelli theorem}\label{section:hr-Torelli}
In this section we prove a higher rank Torelli theorem for the moduli space of semistable parabolic bundles of rank 3, trivial determinant bundle, on a unique parabolic point.

Let $X$ be an elliptic curve and $\infty$ the point corresponding to the neutral element in group structure of $X$. Denote by $\mc{E}_2^1$ the unique non trivial extension 
$$0\to \cO_X\to \mc{E}_2^1\to \cO_X(\infty)\to 0$$
and let $S_1=\PP(\mc{E}_2^1)$ be the associated ruled surface.

\begin{prop}\label{prop:SplusminisS1}
    The surface $\Sigma\cong\Splus\cong\Smin$ is isomorphic to 
    the ruled surface $S_1$.
\end{prop}

\begin{proof}
    Returning to the construction of the universal family in section \ref{sec:universalparabolic}, we observe that the surface $\Splus$ can be deduced as the quotient 
    $$\Splus=\{\lambda=0\}\cup\{\lambda=1\}\cup\{\lambda=\infty\}/\langle \sigma_{12},\sigma_{23} \rangle $$
    of the $3$ special sections by the permutation group. 
    Equivalently, it is given by the quotient
    $$\Splus=\{\lambda=0\}/\langle \sigma_{12}\rangle$$
    of a single section by its isotropy subgroup, which is nothing else than the quotient $\Splus=X\times X/\langle \sigma \rangle$ of the square of $X$ by the permutation $\sigma$ of the two factors. The natural map
    $$\pi:X\times X/\langle \sigma \rangle \to \Jac(X)\ ;\ \{p_1,p_2\}\mapsto L=\cO_X([p_1]+[p_2]-2[\infty])$$
    make $\Splus$ the structure of a ruled surface, since $\pi^{-1}(L)$
    can be viewed as the linear system 
    $\left\vert [p_1]+[p_2]\right\vert\cong\PP^1$. Horizontal curves
    $X\times\{p\}\subset X\times X$ define a family of sections $s_{p}$ of
    $\pi:\Splus\to\Jac(X)$ parametrized by $p\in X$. Since the same family of sections can be defined by vertical curves $\{p\}\times X$, these sections moreover 
    pairwise intersect transversely at a single point: 
    $$\{p_1\}\times X\cap X\times\{p_2\}=(p_1,p_2)\ \ \ \leadsto\ \ \ s_{p_1}\cap s_{p_2}=\{p_1,p_2\}\in\Splus.$$
    We recognize the geometry of $S_1$ as described in \cite{Diaw}
 
    In fact, following Atiyah \cite{atiyah55}, ruled surfaces having $(+1)$-selfintersection sections are either decomposable, i.e.
    $\PP(\cO_X\oplus\cO_X(p))$, or $S_1$. In the former case, the family of $(+1)$-sections form a pencil: they all intersect at the same base point. In our case, a general point $\{p_1,p_2\}\in\Splus$ is the intersection point of two $(+1)$-sections, as for $S_1$.
\end{proof}

Given an elliptic curve $X$ with marked point $p$, we can associate to it, as in the previous section, a pair $(M,\Sigma)$ as follows:
\begin{itemize}
    \item $M=\mathrm{Bun}^{{\bmu}}_{\mc{O}_X}(3,X,p)$ is the moduli space of 
    $\bmu$-stable parabolic bundles for some admissible weight $\bmu$,
    \item $\Sigma\subset M$ is the locus of those parabolic bundles that are stable only 
    for the chamber that contains $\bmu$.
\end{itemize}
Recall that the isomorphism class of $(M,\Sigma)$ does not depend on the choice of $\bmu$.

\begin{thm}\label{Thm: Torelli}
Let $(X,p)$ and $(X',p')$ be elliptic curves, and $(M,\Sigma)$
and $(M',\Sigma')$ the corresponding pairs.
Then $(X,p)\cong (X',p')$ if and only if $(M,\Sigma)\cong (M',\Sigma')$.
\end{thm}

\begin{proof}
Let ${\bmu}$ and $\bmu'$ be two weight vectors. Then for either of the moduli spaces $\mathrm{Bun}^{{\bmu}}_{\mc{O}_X}(3,X,p)$ (and $\mathrm{Bun}^{{\bmu}}_{\mc{O}_{X'}}(3,X',p')$), there are only two choices $M^{\pm}$. By Theorem \ref{thm:MminequalMplus} there is an isomorphism between the compactifying loci $\Psi\vert_{\Smin}:\Smin\stackrel{\sim}{\longrightarrow}\Splus.$ Furthermore, by Proposition \ref{prop:SplusminisS1}, the compactifying loci $\Sigma^\pm$ are isomorphic to the respective ruled surfaces $S_1$ and $S_1^{\prime}$. Recall the Abel-Jacobi map in this case:

\[ \alpha: S_1 \to \Jac(X) , \ \ p_i \oplus p_j \mapsto \mc{O}(p_i\oplus p_j) \]
Respectively, $\alpha^{\prime}: S_1^{\prime} \to \Jac(X')$. 
Note that $\alpha$ and $\alpha'$ are surjective morphisms with fibers being rational curves. By construction the surfaces $S_1$ and $S_1'$ have unique rulings and since the map $\Phi$ is an isomorphism, it must map the fibers of $\alpha$ to $\alpha'$, inducing a birational map $\phi$ between $\Jac(X)$ and $\Jac(X')$. This gives us the following commutative diagram:

\begin{diagram}
     S_1 &\rTo^{\Phi}&S_1^{\prime}\\
     \dTo^{\alpha}&\circlearrowleft&\dTo^{\alpha'}\\
      \Jac(X) &\rTo^{\phi}&\Jac(X')  
         \end{diagram}

Furthermore since $\Jac(X)$ and $\Jac(X')$ are smooth, projective curves, the map $\phi$ is indeed an isomorphism. Note that since $\Jac(X) \simeq X$ and $\Jac(X')\simeq X'$, $\phi$ is an isomorphism of the curves $X,X'$. 
 Moreover, since $X, X'$ are abelian varieties of dimension $1$, $\phi$ maps the point $p$ to $p'$ up to composing with a translation map. In fact, by Theorem \ref{thm:ModularAutomorphism}, we know that the automorphism group is of finite order $9$ and is generated by the $3$-torsion points in $\Jac(X)$. Thus, under the isomorphism $\Phi$, the point $p$ in $X$ is mapped to one of the $9$ points $p' \otimes \mc{O}(t)$, for $t$ a $3$-torsion point.
\end{proof}

 \bibliographystyle{amsalpha}
\bibliography{bibliography}{}

@article{AG,
 author = {Alfaya, David and G{\'o}mez, Tom{\'a}s L.},
 title = {Automorphism group of the moduli space of parabolic bundles over a curve},
 fjournal = {Advances in Mathematics},
 journal = {Adv. Math.},
 issn = {0001-8708},
 volume = {393},
 pages = {127},
 note = {Id/No 108070},
 year = {2021},
 language = {English},
 doi = {10.1016/j.aim.2021.108070},
 zbMATH = {7436487},
 Zbl = {1480.14008}
}

@article{Atiyah55,
 author = {Atiyah, Michael F.},
 title = {Complex fibre bundles and ruled surfaces},
 fjournal = {Proceedings of the London Mathematical Society. Third Series},
 journal = {Proc. Lond. Math. Soc. (3)},
 issn = {0024-6115},
 volume = {5},
 pages = {407--434},
 year = {1955},
 language = {English},
 doi = {10.1112/plms/s3-5.4.407},
 zbMATH = {3279017},
 Zbl = {0174.52804}
}

@article {atiyah57,
    AUTHOR = {Atiyah, M. F.},
     TITLE = {Vector bundles over an elliptic curve},
   JOURNAL = {Proc. London Math. Soc. (3)},
  FJOURNAL = {Proceedings of the London Mathematical Society. Third Series},
    VOLUME = {7},
      YEAR = {1957},
     PAGES = {414--452},
      ISSN = {0024-6115},
   MRCLASS = {14.55 (14.20)},
  MRNUMBER = {0131423},
MRREVIEWER = {F. Hirzebruch},
       DOI = {10.1112/plms/s3-7.1.414},
       URL = {http://dx.doi.org/10.1112/plms/s3-7.1.414},
}

@article{atiyah_2,
 author = {Atiyah, Michael F.},
 title = {Complex analytic connections in fibre bundles},
 fjournal = {Transactions of the American Mathematical Society},
 journal = {Trans. Am. Math. Soc.},
 issn = {0002-9947},
 volume = {85},
 pages = {181--207},
 year = {1957},
 doi = {10.2307/1992969},
 zbMATH = {3128044},
 Zbl = {0078.16002}
}

@article{BBR,
 author = {Balaji, V. and Biswas, Indranil and Del Ba{\~n}o Rollin, Sebastian},
 title = {A {Torelli} type theorem for the moduli space of parabolic vector bundles over curves.},
 fjournal = {Mathematical Proceedings of the Cambridge Philosophical Society},
 journal = {Math. Proc. Camb. Philos. Soc.},
 issn = {0305-0041},
 volume = {130},
 number = {2},
 pages = {269--280},
 year = {2001},
 language = {English},
 doi = {10.1017/S0305004100004916},
 zbMATH = {1590546},
 Zbl = {1063.14042}
}

@article{BGHL,
 author = {Biswas, Indranil and G{\'o}mez, Tom{\'a}s L. and Hoffmann, Norbert and Logares, Marina},
 title = {Torelli theorem for the {Deligne}-{Hitchin} moduli space},
 fjournal = {Communications in Mathematical Physics},
 journal = {Commun. Math. Phys.},
 issn = {0010-3616},
 volume = {290},
 number = {1},
 pages = {357--369},
 year = {2009},
 language = {English},
 doi = {10.1007/s00220-009-0831-3},
 zbMATH = {5655475},
 Zbl = {1184.32003}
}

@article{Biswas-Hoffmann,
 author = {Biswas, Indranil and Hoffmann, Norbert},
 title = {A {Torelli} theorem for moduli spaces of principal bundles over a curve},
 fjournal = {Annales de l'Institut Fourier},
 journal = {Ann. Inst. Fourier},
 issn = {0373-0956},
 volume = {62},
 number = {1},
 pages = {87--106},
 year = {2012},
 language = {English},
 doi = {10.5802/aif.2700},
 zbMATH = {6064513},
 Zbl = {1268.14010}
}

@article{BGM,
 author = {Biswas, Indranil and G{\'o}mez, Tomas L. and Mu{\~n}oz, Vicente},
 title = {Automorphisms of moduli spaces of symplectic bundles},
 fjournal = {International Journal of Mathematics},
 journal = {Int. J. Math.},
 issn = {0129-167X},
 volume = {23},
 number = {5},
 pages = {1250052, 27},
 year = {2012},
 language = {English},
 doi = {10.1142/S0129167X12500528},
 zbMATH = {6037659},
 Zbl = {1246.14048}
}

@article{Boden-Yokogawa-rat,
 author = {Boden, Hans U. and Yokogawa, K{\^o}ji},
 title = {Rationality of moduli spaces of parabolic bundles},
 fjournal = {Journal of the London Mathematical Society. Second Series},
 journal = {J. Lond. Math. Soc., II. Ser.},
 issn = {0024-6107},
 volume = {59},
 number = {2},
 pages = {461--478},
 year = {1999},
 language = {English},
 doi = {10.1112/S0024610799007061},
 zbMATH = {1350191},
 Zbl = {1023.14025}
}

@article{Boden-Yokogawa,
 author = {Boden, Hans U. and Yokogawa, K{\^o}ji},
 title = {Moduli spaces of parabolic {Higgs} bundles and parabolic {{\(K(D)\)}} pairs over smooth curves. {I}},
 fjournal = {International Journal of Mathematics},
 journal = {Int. J. Math.},
 issn = {0129-167X},
 volume = {7},
 number = {5},
 pages = {573--598},
 year = {1996},
 language = {English},
 doi = {10.1142/S0129167X96000311},
 zbMATH = {958513},
 Zbl = {0883.14012}
}

@article{Caporaso-Viviani,
 author = {Caporaso, Lucia and Viviani, Filippo},
 title = {Torelli theorem for stable curves},
 fjournal = {Journal of the European Mathematical Society (JEMS)},
 journal = {J. Eur. Math. Soc. (JEMS)},
 issn = {1435-9855},
 volume = {13},
 number = {5},
 pages = {1289--1329},
 year = {2011},
 language = {English},
 doi = {10.4171/JEMS/281},
 zbMATH = {5947475},
 Zbl = {1230.14037}
}

@article{Diaw,
 author = {Diaw, Arame},
 title = {Geometry of the stable ruled surface over an elliptic curve},
 fjournal = {Bulletin of the Brazilian Mathematical Society. New Series},
 journal = {Bull. Braz. Math. Soc. (N.S.)},
 issn = {1678-7544},
 volume = {52},
 number = {3},
 pages = {645--662},
 year = {2021},
 language = {English},
 doi = {10.1007/s00574-020-00224-7},
 zbMATH = {7375733},
 Zbl = {1467.14093}
}

@article{Maycol,
 author = {Luza, M. Falla},
 title = {On the density of second order differential equations without algebraic solutions on {{\(\mathbb{P}^2\)}}},
 fjournal = {Publicacions Matem{\`a}tiques},
 journal = {Publ. Mat., Barc.},
 issn = {0214-1493},
 volume = {55},
 number = {1},
 pages = {163--183},
 year = {2011},
 language = {English},
 doi = {10.5565/PUBLMAT_55111_08},
 url = {ddd.uab.cat/record/65203},
 zbMATH = {5849332},
 Zbl = {1209.37056}
}

@article{FJ,
 author = {Fassarella, Thiago and Justo, Luana},
 title = {A {Torelli} theorem for moduli spaces of parabolic vector bundles over an elliptic curve},
 fjournal = {Proceedings of the American Mathematical Society},
 journal = {Proc. Am. Math. Soc.},
 issn = {0002-9939},
 volume = {150},
 number = {5},
 pages = {1849--1863},
 year = {2022},
 language = {English},
 doi = {10.1090/proc/15937},
 zbMATH = {7487879},
 Zbl = {1490.14025}
}

@article{Fischer-Grauert-65,
 author = {Fischer, W. and Grauert, Hans},
 title = {Lokal-triviale {Familien} kompakter komplexer {Mannigfaltigkeiten}},
 fjournal = {Nachrichten der Akademie der Wissenschaften in G{\"o}ttingen. II. Mathematisch-Physikalische Klasse},
 journal = {Nachr. Akad. Wiss. G{\"o}tt., II. Math.-Phys. Kl.},
 issn = {0065-5295},
 volume = {1965},
 pages = {89--94},
 year = {1965},
 language = {German},
 zbMATH = {3219133},
 Zbl = {0135.12601}
}

@article{ Basu-Dan-Kaur,
 author = {Basu, Suratno and Dan, Ananyo and Kaur, Inder},
 title = {Degeneration of intermediate {Jacobians} and the {Torelli} theorem},
 fjournal = {Documenta Mathematica},
 journal = {Doc. Math.},
 issn = {1431-0635},
 volume = {24},
 pages = {1739--1767},
 year = {2019},
 language = {English},
 doi = {10.25537/dm.2019v24.1739-1767},
 zbMATH = {7117265},
 Zbl = {1442.14036}
}

@book{LePotier97,
 author = {Le Potier, Joseph},
 title = {Lectures on vector bundles},
 fseries = {Cambridge Studies in Advanced Mathematics},
 series = {Camb. Stud. Adv. Math.},
 volume = {54},
 isbn = {0-521-48182-1},
 year = {1997},
 publisher = {Cambridge: Cambridge University Press},
 language = {English},
 zbMATH = {996200},
 Zbl = {0872.14003}
}

@article{Maruyama,
 author = {Maruyama, Masaki},
 title = {On automorphism groups of ruled surfaces},
 fjournal = {Journal of Mathematics of Kyoto University},
 journal = {J. Math. Kyoto Univ.},
 issn = {0023-608X},
 volume = {11},
 pages = {89--112},
 year = {1971},
 doi = {10.1215/kjm/1250523688},
 zbMATH = {3339067},
 Zbl = {0213.47803}
}

@article{M-Y,
 author = {Maruyama, M. and Yokogawa, K.},
 title = {Moduli of parabolic stable sheaves},
 fjournal = {Mathematische Annalen},
 journal = {Math. Ann.},
 issn = {0025-5831},
 volume = {293},
 number = {1},
 pages = {77--99},
 year = {1992},
 language = {English},
 doi = {10.1007/BF01444704},
 url = {https://eudml.org/doc/164946},
 zbMATH = {4215906},
 Zbl = {0735.14008}
}

@article{M-S,
 author = {Mehta, V. B. and Seshadri, C. S.},
 title = {Moduli of vector bundles on curves with parabolic structures},
 fjournal = {Mathematische Annalen},
 journal = {Math. Ann.},
 issn = {0025-5831},
 volume = {248},
 pages = {205--239},
 year = {1980},
 language = {English},
 doi = {10.1007/BF01420526},
 url = {https://eudml.org/doc/163379},
 zbMATH = {3710299},
 Zbl = {0454.14006}
}

@article{Mumford-Newstead,
 author = {Mumford, D. and Newstead, P.},
 title = {Periods of a moduli space of bundles on curves},
 fjournal = {American Journal of Mathematics},
 journal = {Am. J. Math.},
 issn = {0002-9327},
 volume = {90},
 pages = {1200--1208},
 year = {1968},
 language = {English},
 doi = {10.2307/2373296},
 url = {nrs.harvard.edu/urn-3:HUL.InstRepos:3597250},
 zbMATH = {3279019},
 Zbl = {0174.52902}
}

@article{N-S,
 author = {Narasimhan, M. S. and Seshadri, C. S.},
 title = {Stable and unitary vector bundles on a compact {Riemann} surface},
 fjournal = {Annals of Mathematics. Second Series},
 journal = {Ann. Math. (2)},
 issn = {0003-486X},
 volume = {82},
 pages = {540--567},
 year = {1965},
 language = {English},
 doi = {10.2307/1970710},
 zbMATH = {3272427},
 Zbl = {0171.04803}
}

@article{NFV,
 author = {Fern{\'a}ndez Vargas, N{\'e}stor},
 title = {Geometry of the moduli of parabolic bundles on elliptic curves},
 fjournal = {Transactions of the American Mathematical Society},
 journal = {Trans. Am. Math. Soc.},
 issn = {0002-9947},
 volume = {374},
 number = {5},
 pages = {3025--3052},
 year = {2021},
 language = {English},
 doi = {10.1090/tran/7330},
 zbMATH = {7331115},
 Zbl = {1467.14084}
}

@article{S-T,
 author = {Szenes, Andr{\'a}s and Trapeznikova, Olga},
 title = {The parabolic {Verlinde} formula: iterated residues and wall-crossings},
 fjournal = {Geometry \& Topology},
 journal = {Geom. Topol.},
 issn = {1465-3060},
 volume = {28},
 number = {5},
 pages = {2259--2311},
 year = {2024},
 language = {English},
 doi = {10.2140/gt.2024.28.2259},
 zbMATH = {7927930},
 Zbl = {1558.14024}
}

@article{Torelli,
 author = {Torelli, R.},
 title = {Sulle serie algebriche semplicemente infinite di gruppi di punti appartenenti a una curva algebrica.},
 fjournal = {Accademia dei Lincei, Rendiconti, V. Serie},
 journal = {Rom. Acc. L. Rend. (5)},
 issn = {0001-4435},
 volume = {22},
 number = {1},
 pages = {772--775},
 year = {1913},
 language = {Italian},
 zbMATH = {2622396},
 JFM = {44.0655.02}
}

@article{Tu,
 author = {Tu, Loring W.},
 title = {Semistable bundles over an elliptic curve},
 fjournal = {Advances in Mathematics},
 journal = {Adv. Math.},
 issn = {0001-8708},
 volume = {98},
 number = {1},
 pages = {1--26},
 year = {1993},
 language = {English},
 doi = {10.1006/aima.1993.1011},
 zbMATH = {221237},
 Zbl = {0786.14021}
}

@article{Weil,
 author = {Weil, A.},
 title = {G{\'e}n{\'e}ralisation des fonctions ab{\'e}liennes.},
 fjournal = {Journal de Math{\'e}matiques Pures et Appliqu{\'e}es. Neuvi{\`e}me S{\'e}rie},
 journal = {J. Math. Pures Appl. (9)},
 issn = {0021-7824},
 volume = {17},
 pages = {47--87},
 year = {1938},
 language = {French},
 zbMATH = {2515671},
 JFM = {64.0361.02}
}

\end{document}